\documentclass{amsart}
\usepackage{amssymb}
\usepackage{amscd}
\usepackage[latin1]{inputenc}
\usepackage{amsmath}
\usepackage{amsfonts}
\usepackage{latexsym}
\usepackage{color}
\usepackage{verbatim}
\usepackage{graphicx}
\usepackage{epsfig}
\newtheorem{theorem}{Theorem}[section]
\newtheorem{corollary}[theorem]{Corollary}
\newtheorem{lemma}[theorem]{Lemma}

\newtheorem{proposition}[theorem]{Proposition}
\newtheorem{fact}[theorem]{Fact}

\theoremstyle{definition}
\newtheorem{definition}[theorem]{Definition}
\newtheorem{example}[theorem]{Example}
\newtheorem{remark}[theorem]{Remark}

\newcommand{\zz}{\mathbb{Z}}

\begin{document}
\title[Shifts of finite type with nearly full entropy]{Shifts of finite type with nearly full entropy}
\begin{abstract}
For any fixed alphabet $A$, the maximum topological entropy of a $\mathbb{Z}^d$ subshift with alphabet $A$ is obviously $\log |A|$. We study the class of nearest neighbor $\mathbb{Z}^d$ shifts of finite type which have topological entropy very close to this maximum, and show that they have many useful properties. Specifically, we prove that for any $d$, there exists $\beta_d$ such that for any nearest neighbor $\mathbb{Z}^d$ shift of finite type $X$ with alphabet $A$ for which $(\log |A|) - h(X) < \beta_d$, $X$ has a unique measure of maximal entropy $\mu$. Our values of $\beta_d$ decay polynomially (like $O(d^{-17})$), and we prove that the sequence must decay at least polynomially (like $d^{-0.25+o(1)}$). We also show some other desirable properties for such $X$, for instance that the topological entropy of $X$ is computable and that $\mu$ is isomorphic to a Bernoulli measure. Though there are other sufficient conditions in the literature (see \cite{BS1}, \cite{H}, \cite{MP}) which guarantee a unique measure of maximal entropy for $\mathbb{Z}^d$ shifts of finite type, this is (to our knowledge) the first such condition which makes no reference to the specific adjacency rules of individual letters of the alphabet. 
\end{abstract}
\date{}
\author{Ronnie Pavlov}
\address{Ronnie Pavlov\\
Department of Mathematics\\
University of Denver\\
2360 S. Gaylord St.\\
Denver, CO 80208}
\email{rpavlov@du.edu}
\urladdr{www.math.du.edu/$\sim$rpavlov/}
\thanks{}
\renewcommand{\subjclassname}{MSC 2000}
\maketitle

\section{Introduction}
\label{intro}

A dynamical system consists of a space $X$ endowed with some sort of structure, along with a $G$-action $(T_g)$ on the space for some group $G$ which preserves that structure. (For our purposes, $G$ will always be $\mathbb{Z}^d$ for some $d$.) Two examples are measurable dynamics, where $X$ is a probability space and $T_g$ is a measurable family of measure-preserving maps, and topological dynamics, where $X$ is a compact space and the $T_g$ are a continuous family of homeomorphisms. In each setup, when $G$ is an amenable group, (a class of groups which includes $G = \mathbb{Z}^d$) there is an invaluable notion of entropy; measure-theoretic, or metric, entropy in the setup of measurable dynamics, and topological entropy in the setup of topological dynamics. (We postpone rigorous definitions of these and other terms until Section~\ref{defns}.) These two notions are related by the famous Variational Principle, which says that the topological entropy of a topological dynamical system is the supremum of the measure-theoretic entropy over all $(T_g)$-invariant Borel probability measures supported in it. If in addition the system is taken to be expansive, then this supremum is achieved for at least one measure, and any such measures are called measures of maximal entropy. A question of particular interest is when a system supports a unique measure of maximal entropy. This is closely related to the concept of a phase transition in statistical physics, which occurs when a system supports multiple Gibbs measures. (See \cite{H2} for a discussion of this relationship.)

One particular class of topological dynamical systems for which measures of maximal entropy are well-understood are the one-dimensional shifts of finite type, or SFTs. A one-dimensional shift of finite type is defined by a finite set $A$, called the alphabet, and a finite set $\mathcal{F}$ of forbidden words, or finite strings of letters from $A$. The shift of finite type $X$ induced by $\mathcal{F}$ then consists of all $x \in A^{\mathbb{Z}}$ (biinfinite strings of letters from $A$) which do not contain any of the forbidden words from $\mathcal{F}$. The space $A^{\mathbb{Z}}$ is endowed with the (discrete) product topology, and $X$ inherits the induced topology, under which it is a compact metrizable space. The dynamics of a shift of finite type are always given by the $\mathbb{Z}$-action of integer shifts on sequences in $X$. Any one-dimensional SFT which satisfies a mild mixing condition called irreducibility has a unique measure of maximal entropy called the Parry measure, which is just a Markov chain with transition probabilities which can be algorithmically computed. For more on one-dimensional shifts of finite type and their measures of maximal entropy, see \cite{LM}.

Even for these relatively simple models, things become more complicated when one moves to multiple dimensions. A $d$-dimensional SFT is defined analogously to the one-dimensional case: specify the alphabet $A$ and finite set of forbidden ($d$-dimensional) finite configurations $\mathcal{F}$, and define a shift of finite type $X$ induced by $\mathcal{F}$ to be the set of all $x \in A^{\mathbb{Z}^d}$ (infinite $d$-dimensional arrays of letters from $A$) which do not contain any of the configurations from $\mathcal{F}$. The dynamics are now given by the $\mathbb{Z}^d$-action of all shifts by vectors in $\mathbb{Z}^d$. The easiest class of $d$-dimensional SFTs to work with are the nearest neighbor SFTs; a $d$-dimensional SFT $X$ is called nearest neighbor if $\mathcal{F}$ consists entirely of adjacent pairs of letters, meaning that a point's membership in $X$ is based solely on rules about which pairs of adjacent letters are legal in each cardinal direction. A useful illustrative example of a nearest neighbor SFT is the $d$-dimensional hard-core shift $\mathcal{H}_d$, defined by $A = \{0,1\}$, and $\mathcal{F}$ consisting of all configurations made of adjacent pairs of $1$s (in each of the $d$ cardinal directions). Then $X$ consists of all ways of assigning $0$ and $1$ to each site in $\mathbb{Z}^d$ which do not contain two adjacent $1$s.

It turns out that many questions regarding $d$-dimensional SFTs are extremely difficult or intractable. For instance, given only the alphabet $A$ and forbidden list $\mathcal{F}$, the question of whether or not $X$ is even nonempty is algorithmically undecidable! (\cite{Be}, \cite{Wan}) The structure of the set of measures of maximal entropy for multidimensional SFTs is similarly murky; it has been shown, for instance, that not even the strongest topological mixing properties, which are often enough to preclude some of the difficulties found in $d$-dimensional SFTs, imply uniqueness of the measure of maximal entropy. (\cite{BS0}) Even when the measure of maximal entropy is unique, its structure is not necessarily as simple as in the one-dimensional case: it may be a Bernoulli measure (for instance in the case when the SFT is all of $A^{\mathbb{Z}^d}$), but there also exist examples where it is not even measure-theoretically weak mixing. (\cite{BS3})

There are existing conditions in the literature which guarantee uniqueness of the measure of maximal entropy, but many of these require quite strong restrictions on the adjacency rules defining $X$. For instance, it was first shown in \cite{MP} (using the Dobrushin uniqueness criterion) that if the alphabet of a nearest neighbor $d$-dimensional SFT $X$ has a large enough proportion of letters called safe symbols, meaning that they may legally sit next to any letter of the alphabet in any direction, then $X$ has a unique measure of maximal entropy. It was later shown in \cite{H} that if all letters of the alphabet of a nearest neighbor $d$-dimensional SFT $X$ are only ``nearly safe,'' meaning that they can legally sit next to a large enough proportion of the letters in $A$ in any direction, then again $X$ has a unique measure of maximal entropy. Both of these conditions, though useful, have two problems. Firstly, they make reference to combinatorial information about the adjacency rules themselves, rather than more coarse topological information about the system itself. Secondly, they are not very robust conditions; if one takes an SFT satisfying one of these conditions, and then adds a single letter to $A$ with new adjacency rules which do not allow it to sit next to a large portion of $A$, then the conditions are no longer satisfied.

The main focus of this paper is to define a more robust, less combinatorial, condition on multidimensional SFTs which guarantees existence of a unique measure of maximal entropy. Our condition is similar in spirit to the previously mentioned one from \cite{H}, but rather than requiring every single letter of the alphabet to be ``nearly safe,'' i.e. allowed to sit next to a large proportion of the letters in $A$ in any direction, we require only that a large proportion of the letters of the alphabet are ``nearly safe'' in this sense. More specifically, call a nearest neighbor $d$-dimensional SFT $\epsilon$-full if there exists a subset of the alphabet of size at least $(1 - \epsilon)|A|$ consisting of letters which each have at least $(1 - \epsilon)|A|$ legal neighbors in each cardinal direction. Our main result is that for small enough $\epsilon$ (dependent on $d$), every $\epsilon$-full nearest neighbor $d$-dimensional SFT $X$ has a unique measure of maximal entropy $\mu$. We also prove several other desirable properties for such $X$, such as showing that the topological entropy of $X$ is a computable number and that $\mu$ is measure-theoretically isomorphic to a Bernoulli measure.

Somewhat surprisingly, it is easily shown that the $\epsilon$-fullness condition is implied by a condition which makes no mention of adjacency rules whatsoever, namely having topological entropy very close to the log of the alphabet size. Specifically, for any $\epsilon$, there exists $\beta$ for which any nearest neighbor $d$-dimensional SFT with entropy at least $(\log |A|) - \beta$ is $\epsilon$-full. This shows that all of the properties we prove for $\epsilon$-full nearest neighbor $d$-dimensional SFTs are shared by nearest neighbor $d$-dimensional SFTs with entropy close enough to the log of the alphabet size.

We now briefly summarize the layout of the rest of the paper. In Section~\ref{defns}, we give definitions and basic preliminary results required for our arguments. 
In Section~\ref{props}, we state and prove our main result. In Section~\ref{unrelated}, we show that that $\epsilon$-fullness is unrelated to any existing topological mixing conditions in the literature, i.e. it does not imply and is not implied by any of these conditions. 
Section~\ref{comparison} compares our condition with some other sufficient conditions for uniqueness of measure of maximal entropy from the literature, and in Section~\ref{critical}, we discuss the maximal values of $\beta_d$ which still guarantee uniqueness of measure of maximal entropy for all nearest neighbor $d$-dimensional SFTs with entropy at least $(\log |A|) - \beta_d$, in particular proving polynomially decaying upper and lower bounds on these values.


\section{Definitions and preliminaries}
\label{defns}

We begin with some geometric definitions for $\mathbb{Z}^d$. Throughout, $(\vec{e_i})$ represents the standard orthonormal basis of $\mathbb{Z}^d$. We use $d$ to denote the $\ell_{\infty}$ metric on points in $\mathbb{Z}^d$: $d(s,t) := \|s-t\|_{\infty} = \sum_i |s_i - t_i|$. For any sets $S,T \subseteq \mathbb{Z}^d$, we define $d(S,T) := \min_{s \in S, t \in T} d(s,t)$. We say that two sites $s,t \in \mathbb{Z}^d$ are \textbf{adjacent} if $d(s,t) = 1$. We also refer to adjacent sites as {\bf neighbors}, and correspondingly define the neighbor set $N_t$ of any $t \in \mathbb{Z}^d$ as the set of sites in $\mathbb{Z}^d$ adjacent to $t$.

This notion of adjacency gives $\mathbb{Z}^d$ a graph structure, and the notions of \textbf{paths} and \textbf{connected subsets} of $\mathbb{Z}^d$ are defined with this graph structure in mind. The {\bf outer boundary} of a set $S \subseteq \mathbb{Z}^d$, written $\partial S$, is the set of all $t \in \mathbb{Z}^d \setminus S$ adjacent to some $s \in S$. The {\bf inner boundary} of $S$, written $\underline{\partial} S$, is the set of all $s \in S$ adjacent to some $t \in \mathbb{Z}^d \setminus S$. A {\bf closed contour surrounding $S$} is any set of the form $\partial T$ for a connected set $T \subseteq \mathbb{Z}^d$ containing $S$.

\begin{definition}
For any finite alphabet $A$, the $\zz^d$ \textbf{full shift} over $A$ is the set $A^{\zz^d}$, which is viewed as a compact topological space with the (discrete) product topology.
\end{definition}

\begin{definition}
A \textbf{configuration} over $A$ is a member of $A^S$ for some finite $S \subset \zz^d$, which is said to have \textbf{shape} $S$. The set $\bigcup_{S \subset \zz^d, |S|<\infty} A^S$ of all configurations over $A$ is denoted by $A^*$. When $d=1$, a configuration whose shape is an interval of integers is sometimes referred to as a \textbf{word}.
\end{definition}

\begin{definition}
For two configurations $v \in A^S$ and $w \in A^T$ with $S \cap T = \varnothing$, the \textbf{concatenation} of $v$ and $w$, written $vw$, is the configuration on $S \cup T$ defined by $(vw)|_S = v$ and $(vw)|_T = w$. 
\end{definition}

\begin{definition}
The \textbf{$\zz^d$-shift action}, denoted by $\{\sigma_t\}_{t \in \zz^d}$, is the $\zz^d$-action on a full shift $A^{\zz^d}$ defined by $(\sigma_t x)(s) = x(s+t)$ for $s,t \in \zz^d$. 
\end{definition}

\begin{definition}
A \textbf{$\zz^d$ subshift} is a closed subset of a full shift $A^{\zz^d}$ which is invariant under the shift action.
\end{definition}

Each $\sigma_t$ is a homeomorphism on any $\zz^d$ subshift, and so any $\zz^d$ subshift, when paired with the $\zz^d$-shift action, is a topological dynamical system. An alternate definition for a $\zz^d$ subshift is in terms of disallowed configurations; for any set $\mathcal{F} \subset A^*$, one can define the set $X(\mathcal{F}) := \{x \in A^{\zz^d} \ : \ x|_S \notin \mathcal{F} \ \forall \textrm{ finite } S \subset \zz^d\}$. It is well known that any $X(\mathcal{F})$ is a $\zz^d$ subshift, and all $\zz^d$ subshifts are representable in this way. All $\zz^d$ subshifts are assumed to be nonempty in this paper.

\begin{definition}
A \textbf{$\zz^d$ shift of finite type (SFT)} is a $\zz^d$ subshift equal to $X(\mathcal{F})$ for some finite $\mathcal{F}$. If $\mathcal{F}$ is made up of pairs of adjacent letters, i.e. if $\mathcal{F} \subseteq \bigcup_{i=1}^d A^{\{0,\vec{e_i}\}}$, then $X$ is called a \textbf{nearest neighbor} $\zz^d$ SFT. 
\end{definition}

\begin{definition} 
The \textbf{language} of a $\zz^d$ subshift $X$, denoted by $L(X)$, is the set of all configurations which appear in points of $X$. For any finite $S \subset \zz^d$, $L_S(X) := L(X) \cap A^S$, the set of configurations in the language of $X$ with shape $S$. Configurations in $L(X)$ are said to be {\bf globally admissible}.
\end{definition}

\begin{definition}
A configuration $u \in A^S$ is \textbf{locally admissible} for a $\zz^d$ SFT $X = X(\mathcal{F})$ if $x|_T \notin \mathcal{F}$ for all $T \subseteq S$. In other words, $U$ is locally admissible if it does not contain any of the forbidden configurations for $X$. We denote by $LA(X)$ the set of all locally admissible configurations for $X$, and by $LA_S(X)$ the set $LA(X) \cap A^S$ for any finite $S \subset \mathbb{Z}^d$.
\end{definition}

We note that any globally admissible configuration is obviously locally admissible, but the converse is not necessarily true. (In general, a configuration could be locally admissible, but attempting to complete it to all of $\mathbb{Z}^d$ always leads to a forbidden configuration.)

\begin{definition}
For any $\mathbb{Z}^d$ subshift and configuration $w \in L_S(X)$, the \textbf{cylinder set} $[w]$ is the set of all $x \in X$ with $x|_S = w$. We define the \textbf{configuration set} $\langle w \rangle$ to be the set of all configurations $u$ in $L(X)$ with shape containing $S$ for which $u|_S = w$. For any set $C$ of configurations, we use the shorthand notations $[C]$ and $\langle C \rangle$ to refer to $\bigcup_{w \in C} [w]$ and $\bigcup_{w \in C} \langle w \rangle$ respectively. 
\end{definition}

In the following definition and hereafter, for any integers $m < n$, $[m,n]$ denotes the set of integers $\{m, m+1, \ldots, n\}$.

\begin{definition}\label{topent}
The \textbf{topological entropy} of a $\zz^d$ subshift $X$ is
\[
h(X) := \lim_{n_1, \ldots, n_d \rightarrow \infty} \frac{1}{\prod_{i=1}^d n_i} \log |L_{\prod_{i=1}^d [1,n_i]}(X)|.
\]
\end{definition}

We will also need several measure-theoretic definitions.

\begin{definition}\label{totvar}
For any measures $\mu, \nu$ on the same finite probability space $X$, the \textbf{total variational distance} between $\mu$ and $\nu$ is
\[
d(\mu,\nu) := \frac{1}{2} \sum_{x \in X} |\mu(\{x\}) - \nu(\{x\})| = \max_{A \subseteq X} |\mu(A) - \nu(A)|.
\]
\end{definition}

\begin{definition}\label{coupling}
For any $\mu$, $\nu$ measures on probability spaces $X$ and $Y$ respectively, a \textbf{coupling} of $\mu$ and $\nu$ is a measure $\lambda$ on $X \times Y$ whose marginals are $\mu$ and $\nu$; i.e. $\lambda(A \times Y) = \mu(A)$ for all measurable $A \subseteq X$ and $\lambda(X \times B) = \nu(B)$ for all measurable $B \subseteq Y$. If $X = Y$, then an \textbf{optimal coupling} of $\mu$ and $\nu$ is a coupling $\lambda$ which minimizes the probability $\lambda(\{(x,y) \ : \ x \neq y\})$ of disagreement.
\end{definition}

The connection between Definitions~\ref{totvar} and \ref{coupling} is the well-known fact that for any $\mu$ and $\nu$ on the same finite probability space, optimal couplings exist, and the probability of disagreement for an optimal coupling is equal to the total variational distance $d(\mu,\nu)$. 

From now on, any measure $\mu$ on a full shift $A^{\mathbb{Z}^d}$ is assumed to be a Borel probability measure which is shift-invariant, i.e. $\mu(\sigma_t C) = \mu(C)$ for any measurable $C$ and $t \in \mathbb{Z}^d$.

\begin{definition}\label{measent}
For any measure $\mu$ on a full shift $A^{\zz^d}$, the \textbf{measure-theoretic entropy} of $\mu$ is
\[
h(\mu) := \lim_{n_1, \ldots, n_d \rightarrow \infty} \frac{-1}{\prod_{i=1}^d n_i}  \sum_{w \in A^{\prod_{i=1}^d [1,n_i]}} \mu([w]) \log \mu([w]),
\]
where terms with $\mu([w]) = 0$ are omitted from the sum.
\end{definition}

In Definitions~\ref{topent} and \ref{measent}, a subadditivity argument shows that the limits can be replaced by infimums; i.e. for any $n_1, \ldots, n_d$, $\displaystyle h(X) \leq \frac{1}{\prod_{i=1}^d n_i} \log |L_{\prod_{i=1}^d [1,n_i]}(X)|$ and $\displaystyle h(\mu) \leq \frac{-1}{\prod_{i=1}^d n_i}  \sum_{w \in A^{\prod_{i=1}^d [1,n_i]}} \mu([w]) \log \mu([w])$.

\begin{definition}
For any $\zz^d$ subshift $X$, a \textbf{measure of maximal entropy} on $X$ is a measure $\mu$ with support contained in $X$ for which $h(\mu) = h(X)$.
\end{definition}

The classical variational principle (see \cite{Mi} for a proof) says that for any $\zz^d$ subshift $X$, $\sup_{\mu} h(\mu) = h(X)$, where the supremum, taken over all shift-invariant Borel probability measures whose support is contained in $X$, is achieved. Therefore, any $\zz^d$ subshift has at least one measure of maximal entropy. In the specific case when $X$ is a nearest neighbor $\mathbb{Z}^d$ SFT, much is known about the conditional distributions of a measure of maximal entropy.

\begin{definition}\label{MRF}
A measure $\mu$ on $A^{\mathbb{Z}^d}$ is called a {\bf Markov random field} (or {\bf MRF}) if, for any finite $S \subset \mathbb{Z}^d$, any $w \in A^S$, any finite $T \subset \mathbb{Z}^d \setminus S$ s.t. $\partial S \subseteq T$, and any $\delta \in A^T$ with $\mu([\delta]) \neq 0$,
\[
\mu([w] \ | \ [\delta|_{\partial S}]) = \mu([w] \ | \ [\delta]).
\]
\end{definition}
Informally, $\mu$ is an MRF if, for any finite $S \subset \zz^d$, the sites in $S$ and the sites in $\zz^d \setminus (S \cup \partial S)$ are $\mu$-conditionally independent given the sites on $\partial S$. The following characterization of measures of maximal entropy of nearest neighbor $\mathbb{Z}^d$ SFTs is a corollary of the classical Lanford-Ruelle theorem, but the self-contained version proved in \cite{BS1} is useful for our purposes.

\begin{proposition}\label{P1}
{\rm (\cite{BS1}, Proposition 1.20)}
For any nearest neighbor $\mathbb{Z}^d$ SFT $X$, all measures of maximal entropy for $X$ are MRFs, and for any such measure $\mu$ and any finite shape $S \subseteq \mathbb{Z}^d$, the conditional distribution of $\mu$ on $S$ given any $\delta \in L_{\partial S}(X)$ is uniform over all configurations $x \in L_S(X)$ for which $x \delta \in LA(X)$. 
\end{proposition}

In other words, given any nearest neighbor $\zz^d$ SFT $X$, there is a unique set of conditional distributions that any measure of maximal entropy $\mu$ must match up with. However, this does not uniquely determine $\mu$, as there could be several different measures with the same conditional distributions. For any $\delta \in L_{\partial S}(X)$ as in Proposition~\ref{P1}, we denote by $\Lambda^{\delta}$ the common uniform conditional distribution on $S$ given $\delta$ that every measure of maximal entropy $\mu$ must have.

Next, we define some useful conditions for SFTs and measures supported on SFTs from the literature, many of which we will be able to prove for nearest neighbor $\mathbb{Z}^d$ SFTs which are $\epsilon$-full for small enough $\epsilon$.

\begin{definition}
A {\bf measure-theoretic factor map} between two measures $\mu$ on $A^{\mathbb{Z}^d}$ and $\mu'$ on $B^{\mathbb{Z}^d}$ is a measurable function $F : A^{\mathbb{Z}^d} \rightarrow B^{\mathbb{Z}^d}$ which commutes with the shift action (i.e. $F(\sigma_t x) = \sigma_t F(x)$ for all $x \in A^{\mathbb{Z}^d}$) and for which $\mu'(C) = \mu(F^{-1} C)$ for all measurable $C \subseteq B^{\mathbb{Z}^d}$.
\end{definition}

\begin{definition}
A {\bf measure-theoretic isomorphism} is a measure-theoretic factor map which is bijective between sets of full measure in the domain and range.
\end{definition}

\begin{definition}
A measure $\mu$ on $A^{\mathbb{Z}^d}$ is {\bf ergodic} if any measurable set $C$ which is shift-invariant, meaning $\mu(C \triangle \sigma_t C) = 0$ for all $t \in \mathbb{Z}^d$, has measure $0$ or $1$. Equivalently, $\mu$ is ergodic iff for any configurations $u, u'$ over $A$, 
\[
\lim_{n \rightarrow \infty} \frac{1}{(2n+1)^d} \sum_{t \in [-n,n]^d} \mu([u] \cap \sigma_t [u']) = \mu([u]) \mu([u']).
\]
\end{definition}

\begin{definition}
A measure $\mu$ on $A^{\mathbb{Z}^d}$ is {\bf measure-theoretically strong mixing} if for any configurations $u, u'$ over $A$ and any sequence $t_n \in \mathbb{Z}^d$ for which $\|t_n\|_{\infty} \rightarrow \infty$,
\[
\lim_{n \rightarrow \infty} \mu([u] \cap \sigma_{t_n} [u']) = \mu([u]) \mu([u']).
\]
\end{definition}

\begin{definition}
A measure $\mu$ on $A^{\mathbb{Z}^d}$ is {\bf Bernoulli} if it is 
independent and identically distributed over the sites of $\mathbb{Z}^d$. 
\end{definition}

In dynamics, traditionally a measure is also called Bernoulli if it is measure-theoretically isomorphic to a Bernoulli measure. There is an entire hierarchy of measure-theoretic mixing conditions, all of which are useful isomorphism invariants of measures. (See, for instance, \cite{W}.) We will not spend much space here discussing this hierarchy, because Bernoullicity is the strongest of all of them, and we will verify that the unique measure of maximal entropy of $\epsilon$-full nearest neighbor $\mathbb{Z}^d$ SFTs is isomorphic to a Bernoulli measure for sufficiently small $\epsilon$. 

\begin{definition}
A {\bf topological factor map} between two $\mathbb{Z}^d$ subshifts $X$ and $X'$ is a surjective continuous function $F : X \rightarrow X'$ which commutes with the shift action (i.e. $F(\sigma_t x) = \sigma_t F(x)$ for all $x \in A^{\mathbb{Z}^d}$).
\end{definition}




\begin{definition}
A {\bf topological conjugacy} is a bijective topological factor map.
\end{definition}

The next three definitions are examples of topological mixing conditions, which all involve exhibiting multiple globally admissible configurations in a single point, when separated by a large enough distance.

\begin{definition}\label{mixing}
A $\mathbb{Z}^d$ SFT $X$ is {\bf topologically mixing} if for any configurations $u,u' \in L(X)$, there exists $n$ so that $[u] \cap \sigma_t [u'] \neq \varnothing$ for any $t \in \mathbb{Z}^d$ with $\|t\|_{\infty} > n$.
\end{definition}

\begin{definition}\label{block}
A $\mathbb{Z}^d$ SFT $X$ is {\bf block gluing} if there exists $n$ so that for any configurations $u,u' \in L(X)$ with shapes rectangular prisms and any $t \in \mathbb{Z}^d$ for which $u$ and $\sigma_t u$ are separated by distance at least $n$, $[u] \cap \sigma_t [u'] \neq \varnothing$.
\end{definition}

\begin{definition}\label{UFPdef}
A $\mathbb{Z}^d$ SFT $X$ has the {\bf uniform filling property} or {\bf UFP} if there exists $n$ such that for any configuration $u \in L(X)$ with shape a rectangular prism $R = \prod [a_i, b_i]$, and any point $x \in X$, there exists $y \in X$ such that $y|_R = u$, and $y|_{\mathbb{Z}^d \setminus \prod [a_i - n, b_i + n]} = x|_{\mathbb{Z}^d \setminus \prod [a_i - n, b_i + n]}$.
\end{definition}

All of these conditions are invariant under topological conjugacy. Note the subtle difference in the definitions: Definitions~\ref{block} and \ref{UFPdef} require a uniform distance which suffices to mix between all pairs of configurations of a certain type, whereas Definition~\ref{mixing} allows this distance to depend on the configurations. In general, standard topological mixing is not a very strong condition for $\mathbb{Z}^d$ SFTs; usually a stronger condition involving a uniform mixing length such as block gluing or UFP is necessary to prove interesting results. (See \cite{BPS} for a detailed description of a hierarchy of topological mixing conditions for $\mathbb{Z}^d$ SFTs.)

The final topological properties that we will show for $\epsilon$-full SFTs for small $\epsilon$ do not quite fit into the topological mixing hierarchy. The first involves modeling measure-theoretic dynamical systems within a subshift.

\begin{definition}
A $\mathbb{Z}^d$ subshift $X$ is a {\bf measure-theoretic universal model} if for any $\mathbb{Z}^d$ ergodic measure-theoretic dynamical system $(Y, \mu, (T_t)_{t \in \mathbb{Z}^d})$, there exists a measure $\nu$ on $X$ so that $(X, \nu, (\sigma_t)_{t \in \mathbb{Z}^d}) \cong (Y, \mu, (T_t)_{t \in \mathbb{Z}^d})$.
\end{definition}

It was shown in \cite{RS} that any $\mathbb{Z}^d$ SFT with the UFP is a measure-theoretic universal model.



We also need a definition from computability theory.

\begin{definition}
A real number $\alpha$ is \textbf{computable in time} $f(n)$ if there exists a Turing machine which, on input $n$, outputs a pair $(p_n, q_n)$ of integers such that $|\frac{p_n}{q_n} - \alpha| < \frac{1}{n}$, and if this procedure takes less than $f(n)$ operations for every $n$. We say that $\alpha$ is \textbf{computable} if it is computable in time $f(n)$ for some function $f(n)$.
\end{definition}


For an introduction to computability theory, see \cite{Ko}. The relationship between multidimensional symbolic dynamics and computability theory has been the subject of much work in recent years, but is still not completely understood. One foundational result is from \cite{HM}, where it is shown that a real number is the entropy of some $\mathbb{Z}^d$ SFT for any $d > 1$ if and only if it has a property called right recursive enumerability, which is strictly weaker than computability and which we do not define here. It is also shown in \cite{HM} that if a $\mathbb{Z}^d$ SFT has the uniform filling property, then its entropy is in fact computable.

\

We conclude this section by finally defining $\epsilon$-fullness of a nearest neighbor $\mathbb{Z}^d$ SFT and showing its connection to entropy.

\begin{definition} 
For any $\epsilon > 0$, we say that a nearest neighbor $\mathbb{Z}^d$ SFT $X$ with alphabet $A$ is {\bf $\epsilon$-full} if $A$ can be partitioned into sets $G$ (good letters) and $B$ (bad letters) with the properties that

(i) $|G| > (1-\epsilon)|A|$

(ii) $\forall g \in G, i \in [1,d], \tau \in \{\pm 1\}$, the set of legal neighbors of $g$ in the $\tau \vec{e_i}$-direction has cardinality greater than $(1-\epsilon) |A|$. 
\end{definition}

We first show some useful technical properties for $\epsilon$-full nearest neighbor $\mathbb{Z}^d$ SFTs with small $\epsilon$.

\begin{lemma}\label{manyfillings}
If $X$ is $\epsilon$-full for $\epsilon < \frac{1}{2d+2}$, then for any locally admissible configuration $w$ with shape $S$ with $w|_{\underline{\partial} S} \in G^{\underline{\partial} S}$ and any $t \in \mathbb{Z}^d \setminus S$, there exists a nonempty subset $G'$ of $G$ with cardinality greater than $|A|(1 - (2d+1)\epsilon)$ so that for any $g' \in G'^{\{t\}}$, the concatenation $w g'$ is locally admissible.
\end{lemma}

\begin{proof}

Since $\epsilon < \frac{1}{2d+2}$, $|A|(1 - (2d+1)\epsilon) \geq |A| \epsilon$. We note that if $|A| \epsilon < 1$, then $\epsilon$-fullness of $X$ implies that $G = A$ and $X$ is a full shift, in which case the lemma is trivial. So, we can assume that $|A|(1 - (2d+1) \epsilon) \geq 1$. Define $N = N_t \cap S$, and note that $N \subseteq \underline{\partial} S$. For any $a \in A$, as long as $a$ can appear legally next to each of the at most $2d$ letters in $w|_N$, the concatenation $wa$ is locally admissible. Each letter in $w|_N$ is a $G$-letter, and so by $\epsilon$-fullness, for each $t \in N$, the set of letters which can appear legally at $t$ adjacent to $w(t)$ has cardinality at least $|A| (1 - \epsilon)$, and so there are at least $|A| (1 - 2d \epsilon)$ letters in $A$ for which $wa$ is locally admissible. Since $|G| > |A| (1-\epsilon)$, at least $|A| (1 - (2d+1) \epsilon)$ of these letters are in $G$, and we are done.

\end{proof}

\begin{corollary}\label{LA=GA}
If $X$ is $\epsilon$-full and $\epsilon < \frac{1}{2d+2}$, then any locally admissible configuration $w$ with shape $S$ with $w|_{\underline{\partial} S}$ consisting only of $G$-letters is also globally admissible. In particular, $w$ can be extended to a point of $X$ by appending only $G$-letters to $w$.
\end{corollary}

\begin{proof}

As before, we can reduce to the case where $|A|(1 - (2d+1) \epsilon) \geq 1$. Suppose $w$ is a locally admissible configuration with shape $S$ s.t. $w|_{\underline{\partial} S}$ consists only of $G$-letters, and arbitrarily order the sites in $\mathbb{Z}^d \setminus S$ as $s_i$, $i \in \mathbb{N}$. We claim that for any $n$, there exists a locally admissible configuration $w_n$ with shape $S \cup \bigcup_{i=1}^{n} \{s_i\}$ such that $w_n|_S = w$, $w_n|_{(\underline{\partial} S) \cup \bigcup_{i=1}^{n} \{s_i\}}$ consists only of $G$-letters, and each $w_n$ is a subconfiguration of $w_{n+1}$.

The proof is by induction: the existence of $w_1$ is obvious by applying Lemma~\ref{manyfillings} to $w$ and $s_1$, and for any $n$, if we assume the existence of $w_n$, the existence of $w_{n+1}$ comes from applying Lemma~\ref{manyfillings} to $w_n$ and $s_{n+1}$, along with the observation that clearly $\underline{\partial}(S \cup \bigcup_{i=1}^n \{s_i\}) \subseteq (\underline{\partial} S) \cup \bigcup_{i=1}^n \{s_i\}$.

Then the $w_n$ approach a limit point $x \in G^{\mathbb{Z}^d}$, which is in $X$ since each $w_n$ was locally admissible. Since $w$ was a subconfiguration of each $w_n$, it is a subconfiguration of $x$, and so $w \in L(X)$.

\end{proof}

Surprisingly, the $\epsilon$-fullness property is closely related to a simpler property which can be stated without any reference to adjacency rules, i.e. having entropy close to the log of the alphabet size.

\begin{theorem}\label{equiv}
For any $\epsilon > 0$ and $d$, there exists a $\beta = \beta(\epsilon,d)$ so that for a nearest neighbor $\mathbb{Z}^d$ SFT $X$ with alphabet $A$,
\[
h(X) > (\log |A|) - \beta \Longrightarrow X\textrm{ is } \epsilon-\textrm{full}.
\]
Also, for any $\beta > 0$ and $d$, there exists an $\epsilon = \epsilon(\beta,d)$ so that 
\[
X\textrm{ is } \epsilon-\textrm{full} \Longrightarrow h(X) > (\log |A|) - \beta.
\]
\end{theorem}

\begin{proof}
Fix any $d$ and $\epsilon > 0$, and suppose that $X$ is not $\epsilon$-full. This implies that if we define $B$ to be the set of $b \in A$ for which there exists $i \in [1,d]$ and $\tau \in \{\pm 1\}$ so that there are at least $\epsilon |A|$ letters which cannot follow $b$ in the $\tau \vec{e_i}$-direction, then $|B| > \epsilon |A|$. (Otherwise, taking $G$ to be $B^c$ would show that $X$ is $\epsilon$-full.)

Then, there exist $\tau \in \{\pm 1\}$, $i \in [1,d]$ and a set $B_i \subset B$ with $|B_i| > \frac{\epsilon}{2d} |A|$ so that for each $b \in B_i$, there are at least $\epsilon |A|$ letters which cannot follow $b$ in the $\vec{e_i}$-direction. This implies that there are at least $|B_i| \epsilon |A| > \frac{\epsilon^2}{2d} |A|^2$ configurations with shape $\{0\} \cup \{\tau \vec{e_i}\}$ which are not in $L(X)$, and so $|L_{\{0\} \cup \{\tau \vec{e_i}\}}(X)| < |A|^2 \left(1 - \frac{\epsilon^2}{2d}\right)$. Then
\[
h(X) \leq \frac{1}{2} \log |L_{\{0\} \cup \{\vec{e_i}\}}(X)| < \log |A| - \frac{1}{2} \log \frac{2d}{2d - \epsilon^2},
\]
and so taking $\beta(\epsilon,d) = \frac{1}{2} \log \frac{2d}{2d - \epsilon^2}$ proves the first half of the theorem. For future reference, we note that $\beta(\epsilon,d) = \frac{1}{2} \log \frac{2d}{2d - \epsilon^2} = - \frac{1}{2} \log \left(1 - \frac{\epsilon^2}{2d}\right) > -\frac{1}{2} \left(-\frac{\epsilon^2}{2d}\right) = \frac{\epsilon^2}{4d}$.

Now fix any $\epsilon > 0$, and suppose that $X$ is $\epsilon$-full. Then, for any $n$, we bound from below the size of $L_{[1,n]^d}(X)$. Construct configurations in the following way: order the sites in $[1,n]^d$ lexicographically. Then fill the first site in $[1,n]^d$ with any $G$-letter. Fill the second site with any $G$-letter which can legally appear next to the first placed $G$-letter, and continue in this fashion, filling the sites in order with $G$-letters, and each time placing any legal choice given the letters which have already been placed. By Lemma~\ref{manyfillings}, at each step we will have at least $|A|(1 - (d+1)\epsilon)$ choices. We can create more than $(|A|(1 - (d+1)\epsilon))^{n^d}$ configurations in this way, and each one is in $L(X)$ by Lemma~\ref{LA=GA}. This means that for any $n$,
\[
|L_{[1,n]^d}(X)| > (|A|(1 - (d+1)\epsilon))^{n^d},
\]
which implies that $h(X) \geq \log(|A|(1 - (d+1)\epsilon)) > \log |A| + \log(1 - (d+1)\epsilon)$. Then taking $\epsilon(\beta,d) = \frac{1 - e^{-\beta}}{d+1}$ proves the second half of the theorem.

\end{proof}


We will informally refer to nearest neighbor $\mathbb{Z}^d$ SFTs with topological entropy close to $\log |A|$ as having ``nearly full entropy.'' By Theorem~\ref{equiv}, this condition is equivalent to being $\epsilon$-full for small $\epsilon$, and so we will often use the terms ``$\epsilon$-full for small $\epsilon$'' and ``nearly full entropy'' somewhat interchangeably in the sequel.

Before stating and proving our main results, here are some examples of nearest neighbor $\mathbb{Z}^d$ SFTs which are $\epsilon$-full for small $\epsilon$.

\begin{example}\label{onepoint}
Take $X$ to have alphabet $A = \{0,1,\ldots,n\}$, and the only adjacency rule is that any neighbor of a $0$ must also be a $0$. Then $X$ is just the union of the full shift on $\{1,\ldots,n\}$ and a fixed point of all $0$s. Clearly $X$ is $\epsilon$-full for $\epsilon < \frac{1}{n + 1}$.
\end{example}

\begin{example}\label{vertlines}
Take $X$ to have alphabet $A = \{0,1,\ldots,n\}$, and the only adjacency rule is that a $0$ can only appear above and below other $0$s. Then $X$ consists of points whose columns are either sequences on $\{1,\ldots,n\}$ (with no restrictions on which rows can appear) or all $0$s. Again, clearly $X$ is $\epsilon$-full for $\epsilon < \frac{1}{n+1}$.
\end{example}

\begin{example}\label{fullshift}
Take $X$ to be the full shift on $A = \{0,1,\ldots,n\}$. Then trivially $X$ is $\epsilon$-full for any $\epsilon > 0$. For the purposes of this example though, think of $A$ as being partitioned into $G = \{1,\ldots,n\}$ and $B = \{0\}$, which would demonstrate that $X$ is $\epsilon$-full for $\epsilon < \frac{1}{n+1}$.
\end{example}

These examples illustrate the different ways in which $B$-letters can coexist with $G$-letters, which is the unknown quantity in the description of $\epsilon$-full SFTs. 
In Examples~\ref{onepoint} and \ref{vertlines}, the existence of a $B$-letter forces the existence of an infinite component of $B$-letters. It turns out that in such examples, $B$-letters are rare in ``most'' configurations of $X$; in particular, they have zero measure for any measure of maximal entropy. In contrast, Example~\ref{fullshift} clearly has a unique Bernoulli measure of maximal entropy, whose support contains all configurations (including those with $B$-letters). However, for large $n$, in ``most'' configurations the $B$-letters appear with the small frequency $\frac{1}{n}$. The dichotomy is that for measures of maximal entropy, either $B$-letters can only appear within infinite clusters of $B$-letters (and then have zero measure), or $B$-letters ``coexist peacefully'' with $G$-letters, and appear in most configurations, albeit with frequency less than $\epsilon$.

\section{Properties of $\epsilon$-full/nearly full entropy SFTs}
\label{props}

The goal of this section is to prove the following theorem.


\begin{theorem}\label{mainresult}
For any $d$, and $\epsilon_d := \frac{1}{3^3 2^{11} d^8}$, 
any $\epsilon_d$-full nearest neighbor $\mathbb{Z}^d$ SFT $X$ has the following properties:

{\rm (A)} $X$ has a unique measure of maximal entropy $\mu$

{\rm (B)} $h(X)$ is computable in time $e^{O(n^d)}$ 

{\rm (C)} $X$ is a measure-theoretic universal model

{\rm (D)} $\mu$ is measure-theoretically isomorphic to a Bernoulli measure


\end{theorem}

By Theorem~\ref{equiv}, all such properties also hold for nearest neighbor $\mathbb{Z}^d$ SFTs with topological entropy close enough to the logarithm of their alphabet size:

\begin{theorem}\label{mainresult2}
For any $d$ and $\beta_d = \frac{\epsilon_d^2}{4d} = \frac{1}{3^6 2^{24} d^{17}}$, any nearest neighbor $\mathbb{Z}^d$ SFT $X$ with alphabet $A$ for which $h(X) > (\log |A|) - \beta_d$ has properties {\rm(A)-(D)} from Theorem~\ref{mainresult}.
\end{theorem}


The next lemma will be fundamental to almost all future arguments, and deals with the conditional measure w.r.t. a measure of maximal entropy $\mu$ of a configuration consisting only of $B$-letters given a boundary configuration. We would like to be able to say that such configurations always have low conditional probability, but this depends on the boundary. For instance, in the SFT of Example~\ref{vertlines}, conditioning on a boundary $\delta \in A^{\partial [-n,n]^2}$ for which $\delta(0,-n) = \delta(0,n) = 0$ actually forces an entire column of $0$s! For this reason, we for now deal only with the case where we condition on a boundary consisting only of $G$-letters.

\begin{lemma}\label{unlikelyB}
For any $\epsilon < \frac{1}{4d+6}$, any $\epsilon$-full nearest neighbor $\mathbb{Z}^d$ SFT $X$, any set $S \subseteq \mathbb{Z}^d$, any set $T \subseteq S$, any $\delta \in G^{\partial S}$, and any measure of maximal entropy $\mu$ on $X$, 
\[
\mu([B^T] \ | \ [\delta]) < N^{-|T|},
\]
where $N$ is $\lfloor \frac{1}{2} (\epsilon^{-1} - 4d - 4) \rfloor$.
\end{lemma}

\begin{proof}

Consider any such $\epsilon$, $X$, $S$, $T$, and $\delta$, and define $N = \lfloor \frac{1}{2} (\epsilon^{-1} - 4d - 4) \rfloor$; since $\epsilon < \frac{1}{4d+6}$, $N \geq 1$ and the inequality we wish to prove is nontrivial. As before, we can reduce to the case where $|A| \epsilon \geq 1$ since otherwise $X$ is forced to be a full shift. By $\epsilon$-fullness of $X$, $|G| > |A| (1 - \epsilon) \geq |A| (1 - 2\epsilon) + 1$, and by definition of $N$, $1-2\epsilon > 2 \epsilon(N + 2d + 1)$. Therefore, $|G| > 2 |A| \epsilon (N + 2d + 1) + 1 \geq 2 \lceil |A| \epsilon (N + 2d + 1) \rceil$. We can then partition $G$ into two pieces, call them $G_I$ and $G_B$, each of size at least  $|A| \epsilon (N + 2d + 1)$, and fix any orderings on the elements of $G_I$, $G_B$, and $B$. 

Consider any configuration $u \in L_{S \cup \partial S}(X) \cap \langle B^T \rangle \cap \langle \delta \rangle$, i.e. $u$ is globally admissible with shape $S \cup \partial S$, $u|_T$ consists entirely of $B$-letters, and $u|_S = \delta$. Then the locations of the $B$-letters within $u$ can be partitioned into maximal connected components $C_i(u)$, $1 \leq i \leq k(u)$ (say we order these lexicographically by least element), and we denote the subconfigurations of $u$ occupying these components by $B_i(u) = u|_{C_i}$, $1 \leq i \leq k(u)$. We will now define a family of configurations $f(u) \subseteq L_{S \cup \partial S}(X) \cap \langle \delta \rangle$.

Begin by removing all $B_i(u)$ from $u$, defining a new configuration $v(u)$ with shape $(S \cup \partial S) \setminus \bigcup C_i(u)$ which consists only of $G$-letters. We fill the holes with shapes $C_i(u)$ in order, starting with $C_1(u)$. For each $i$, we order the sites in $C_i(u)$ lexicographically, and choose $G$-letters to fill them, one by one. We will do this in such a way that at each step, regardless of what letters have been assigned, we have $N$ choices of letters to use, and so the total number of configurations we define by filling all holes, $|f(u)|$, will be at least $N^{|T|}$.

Suppose that we wish to fill a site $s \in C_i(u)$, meaning that each $C_j(u)$ for $j<i$ has been filled, and all sites lexicographically less than $s$ in $C_i(u)$ have been filled with $G$-letters. Then, consider all $G$-letters which can legally fill the site $s$ given the letters already assigned within $S \cup \partial S$. Since all letters assigned are $G$-letters, by Lemma~\ref{manyfillings}, there are at least $|A| (1 - (2d+1) \epsilon)$ choices. If $s \in \partial C_i(u)$, then will use only letters from $G_B$, and if $s \in C_i(u) \setminus \partial C_i(u)$, then we will use only letters from $G_I$. In either case, though, since $|G_I|$ and $|G_B|$ are greater than $|A|\epsilon(N+2d+1)$, there are at least $|A| N \epsilon > N|B|$ possible choices. If $u(s)$ was the $k$th letter of $B$ with respect to the previously defined ordering on $B$, then we use any of the $N$ letters between the $((k-1)N + 1)$th and $kN$th letters (inclusive) in either $G_B$ or $G_I$ with respect to the previously defined orderings on these sets. Denote by $f(u)$ the set of all configurations in $L_{S \cup \partial S}(X) \cap \langle \delta \rangle$ obtainable by using this filling algorithm to fill all of the sites of $T$ in order. Now for each site $s \in \bigcup C_i(u)$, each configuration in $f(u)$ has a letter at $s$ which encodes the following information: whether $s$ was a boundary site or an interior site within its $C_i(u)$ (encoded by whether we chose a letter from $G_B$ or $G_I$), and the $B$-letter $u(s)$ which appeared at $s$ in $u$ (encoded by which of the possible letters in $G_B$ or $G_I$ we used).

We now show that for any configurations $u \neq u'$ in $L_{S \cup \partial S}(X) \cap \langle B^T \rangle \cap \langle \delta \rangle$, $f(u)$ and $f(u')$ are disjoint. First, we deal with the case where $k(u) = k(u')$ and $C_i(u) = C_i(u')$ for $1 \leq i \leq k(u) = k(u')$. (Since they are equal, we just write $C_i$ for $C_i(u) = C_i(u')$ and $k$ for $k(u) = k(u')$.) Since $u \neq u'$, $u$ and $u'$ either disagree somewhere outside the union of the $C_i$ or somewhere inside. If there is a disagreement somewhere outside, then since all configurations in $f(u)$ and $f(u')$ agree with $u$ and $u'$ respectively outside the union of the $C_i$, it is obvious that $f(u)$ and $f(u')$ are disjoint. If there is a disagreement inside the union of the $C_i$, then take $j$ minimal so that there is a disagreement in $C_j$, and take $s \in C_j$ the minimal site lexicographically for which $u(s) \neq u'(s)$. For a contradiction, assume that there is a configuration $w$ in $f(u) \cap f(u')$. Since all $C_i$ are identical and since $u$ and $u'$ agree outside the union of the $C_i$, we know that exactly the same sites had been filled, with exactly the same letters, when $w(s)$ was chosen in the filling procedure defining $f(u)$ as when $w(s)$ was chosen in the filling procedure defining $f(u')$. But this is a contradiction; since $u(s) \neq u'(s)$ and the same set of letters was available to fill $s$ in both procedures, the same letter could not possibly have been a legal choice in both procedures.

Now we deal with the case where either $k(u) \neq k'(u)$ or $k(u) = k'(u)$ and $C_i(u) \neq C_i(u')$ for some $1 \leq i \leq k(u) = k(u')$. This implies that either there exists $C_j(u)$ disjoint from all $C_i(u')$ (or the same statement with $u$ and $u'$ reversed), or there exist nonequal $C_j(u)$ and $C_{j'}(u')$ which have nonempty intersection (or the same statement with $u$ and $u'$ reversed). The first case is impossible since by definition, each $C_j(u)$ contains some site in $T$, and each site in $T$ is contained in some $C_i(u')$ (and the same statement is true when $u$ and $u'$ are reversed). Suppose then that there exist $j,j'$ so that $C_j(u) \neq C_{j'}(u')$ and $C_j(u) \cap C_{j'}(u') \neq \varnothing$. Then there exists $s$ which is in the boundary of $C_j(u)$ and the interior of $C_{j'}(u')$, or vice versa. This means that when $s$ is assigned in the filling procedures defining $f(u)$ and $f(u')$, either $w(s)$ must be from $G_B$ in the former case and $G_I$ in the latter, or vice versa. Either way, it ensures that $f(u) \cap f(u') = \varnothing$.

We have shown that all of the sets $f(u)$, $u \in L_{S \cup \partial S}(X) \cap \langle B^T \rangle \cap \langle \delta \rangle$, are disjoint. Since each is a subset of $L_{S \cup \partial S}(X) \cap \langle \delta \rangle$ and each has size at least $N^{|T|}$, we have shown that 
\[
|L_{S \cup \partial S}(X) \cap \langle \delta \rangle| \geq N^{|T|} |L_{S \cup \partial S}(X) \cap \langle B^T \rangle \cap \langle \delta \rangle|.
\]
Recall that since $\mu$ is a measure of maximal entropy for $X$, by Proposition~\ref{P1} it is an MRF with uniform conditional probabilities $\Lambda^{\delta}$. Therefore,
\[
\mu([B^T] \ | \ [ \delta ]) = \Lambda^{\delta}(\langle B^T \rangle) = \frac{|L_{S \cup \partial S}(X) \cap \langle B^T \rangle \cap \langle \delta \rangle|}{|L_{S \cup \partial S}(X) \cap \langle \delta \rangle|} < N^{-|T|}.
\]

\end{proof}

\begin{remark}\label{extrasites}
For future reference, we note that since $\mu$ is an MRF, Lemma~\ref{unlikelyB} remains true if one additionally conditions on some sites outside $S \cup \partial S$, even if these extra sites are taken to be $B$-letters.
\end{remark} 

We are now prepared to prove Theorem~\ref{mainresult}. We fix $d$, define $\epsilon_d = \frac{1}{3^3 2^{11} d^8}$, and consider any $\epsilon_d$-full nearest neighbor $\mathbb{Z}^d$ SFT $X$. We usually suppress the dependence on $d$ and just write $\epsilon = \epsilon_d$ in the sequel.

\

\begin{proof}[Proof of (A)] Recall that for any finite $S \subseteq \mathbb{Z}^d$ and $\delta \in L_{\partial S}(X)$, $\Lambda^{\delta}$ is the conditional distribution on $S$ given $\delta$ associated to any measure of maximal entropy, which is uniformly distributed over all configurations $w \in A^S$ which form a locally admissible configuration when combined with $\delta$. We will show that there is only a single shift-invariant measure $\mu$ with these conditional distributions, implying that there is only a single measure of maximal entropy. Our method is similar to that of \cite{vdBM} in that we construct a coupling of $\Lambda^{\delta}$ and $\Lambda^{\delta'}$ for pairs of boundaries $\delta \neq \delta'$ of large connected shapes, and show that this coupling gives a high probability of agreement far from $\delta$ and $\delta'$, implying that $\Lambda^{\delta}$ and $\Lambda^{\delta'}$ behave similarly far from $\delta$ and $\delta'$. (Informally, the influence of a boundary decays with distance.) 
However, we must begin with the special case where $\delta$ and $\delta'$ consist entirely of $G$-letters. 

Choose any finite connected sets $C, C' \subset \mathbb{Z}^d$ with nonempty intersection, any site $s \in C \cap C'$, and any $\delta \in L_{\partial C}(X)$ and $\delta' \in L_{\partial C'}(X)$ consisting only of $G$-letters. Define $D:= d(s, \partial C \cup \partial C')$. We will construct a coupling $\lambda$ of $\Lambda^{\delta}$ and $\Lambda^{\delta'}$ which gives very small probability to a disagreement at $s$ (when $D$ is large).

Define $\overline{C} = C \cup \partial C$ and $\overline{C'} = C' \cup \partial C'$. Fix any ordering on the set $\overline{C} \cup \overline{C'}$; from now on when we talk about any notion of size for sites in $\overline{C} \cup \overline{C'}$, it is assumed we are speaking of this ordering. For convenience, we will extend configurations on $C$ and $C'$ to configurations on $\overline{C}$ and $\overline{C'}$ respectively by appending $\delta$ and $\delta'$ respectively. Therefore, $\lambda$ will be defined on pairs of configurations $(w_1, w_2)$ where $w_1$ has shape $\overline{C}$ and $w_2$ has shape $\overline{C'}$; the marginalization of $\lambda$ which leads to a true coupling of $\Lambda^{\delta}$ and $\Lambda^{\delta'}$ should be clear. We will define $\lambda$ on one site at a time, assigning values to both $w_1(s)$ and $w_2(s)$ when $s$ is in $\overline{C} \cap \overline{C'}$, and just assigning one of these two values if $s$ is only one of the sets. We use $\zeta_1$ and $\zeta_2$ to denote the (incomplete) configurations on $\overline{C}$ and $\overline{C'}$ respectively at any step. We therefore begin with $\zeta_1 = \delta$ and $\zeta_2 = \delta'$. At any step of the construction, we use $W$ to denote the set of vertices in $\overline{C} \cup \overline{C'}$ on which either $\zeta_1$ or $\zeta_2$ have already received values. (In particular, at the beginning, $W = \partial C \cup \partial C'$.) This means that $\zeta_1$ is always defined on $W \cap \overline{C}$, and $\zeta_2$ is always defined on $W \cap \overline{C'}$. At an arbitrary step of the construction, we choose the next site $s$ on which to assign values in $\zeta_1$ and/or $\zeta_2$ as follows:

\

(i) If there exists any site in $(\overline{C} \cup \overline{C'}) \setminus W$ which is adjacent to a site in $W$ at which either $\zeta_1$ or $\zeta_2$ has been assigned a $B$-letter, then take $s$ to be the smallest such site.

(ii) If (i) does not apply, but there exists a site in $(\overline{C} \cup \overline{C'}) \setminus W$ which is adjacent to a site in $\overline{C} \cap \overline{C'} \cap W$ (i.e. a site at which both $\zeta_1$ and $\zeta_2$ have been defined), and their values disagree, then take $s$ to be the smallest such site.

(iii) If (i) and (ii) do not apply, but there exists a site in $(\overline{C} \cup \overline{C'}) \setminus W$ which is not in $\overline{C} \cap \overline{C'}$, then take $s$ to be the smallest such site.

(iv) If none of (i)-(iii) apply, then take $s$ to be the smallest site in $(\overline{C} \cup \overline{C'}) \setminus W$.

\

Now we are ready to define $\lambda$ on $s$. If $s$ is in $\overline{C}$ but not $\overline{C'}$ (i.e. chosen according to case (iii)), then assign $\zeta_1(s)$ randomly according to the marginalization of the distribution $\Lambda^{\zeta_1}$ to $s$, and if $s$ is in $\overline{C'}$ but not $\overline{C}$, then assign $\zeta_2(s)$ randomly according to the marginalization of the distribution $\Lambda^{\zeta_2}$ to $s$. (Here we are slightly abusing notation: $\Lambda^{\eta}$ is technically only defined for $\eta$ a boundary configuration, and here we may be conditioning on more than a boundary. The meaning should be clear though: $\Lambda^{\zeta_1}$, simply represents the uniform conditional distribution on $A^{\overline{C} \setminus W}$ given $\zeta_1$, and $\Lambda^{\zeta_2}$ is similarly defined.)

If $s \in \overline{C} \cap \overline{C'}$ (i.e. chosen according to case (i) or case (ii)), then assign $\zeta_1(s)$ and $\zeta_2(s)$ according to an optimal coupling of the marginalizations of the distributions $\Lambda^{\zeta_1}$ and $\Lambda^{\zeta_2}$ to $s$. Since $\lambda$ is defined sitewise, and at each step is assigned according to $\Lambda^{\zeta_1}$ in the first coordinate and $\Lambda^{\zeta_2}$ in the second, the reader may check that it is indeed a coupling of $\Lambda^{\delta}$ and $\Lambda^{\delta'}$. The key property of $\lambda$ is the following:

\begin{fact}\label{path}
For any site $s \in \overline{C} \cap \overline{C'}$, $\lambda$-a.s., $w_1(s) \neq w_2(s)$ if and only if there exists a path $\gamma$ from $s$ to $\partial C \cup \partial C'$ contained within $\overline{C} \cap \overline{C'}$ such that for each site $t \in \gamma$, either one of $w_1(t)$ or $w_2(t)$ is a $B$-letter, or $w_1$ and $w_2$ disagree at $t$, i.e. $w_1(t) \neq w_2(t)$.
\end{fact}

\begin{proof}
The ``if'' direction is trivial. For the ``only if'' direction, assume for a contradiction that $w_1(s) \neq w_2(s)$ and that no such path $\gamma$ exists. Then there is a closed contour $\Gamma$ containing $s$ and contained within $\overline{C} \cap \overline{C'}$ so that $w_1|_{\Gamma} = w_2|_{\Gamma} \in G^{\Gamma}$. Denote by $F$ the set of sites inside $\Gamma$. Then regardless of the order of the sites on which $\lambda$ is defined, the first site in $F$ which is assigned is done so by case (iv); since it is the first site in $F$ to be assigned, its neighbors are either unassigned or in $\Gamma$, and so cases (i)-(iii) cannot apply. Call this site $t$.

Consider the state of $\lambda$ when $t$ is assigned under case (iv). The sets of undefined sites for $\zeta_1$ and $\zeta_2$ must be the same (since case (iii) was not applied), and every site in $\overline{C} \cup \overline{C'}$ adjacent to a site in $(\overline{C} \cup \overline{C'}) \setminus W$ must be a location at which $\zeta_1$ and $\zeta_2$ agree. (since case (ii) was not applied) Then the distributions $\Lambda^{\zeta_1}$ and $\Lambda^{\zeta_2}$ are identical. This means that their optimal coupling has support contained in the diagonal, and $\zeta_1(t) = \zeta_2(t)$ $\lambda$-a.s. It is then easy to see that $\lambda$-a.s., we remain in case (iv) for the remainder of the construction. Therefore, $\lambda$-a.s., $\zeta_1$ and $\zeta_2$ agree on all of $F$. Since $s \in F$, this clearly contradicts $w_1(s) \neq w_2(s)$.

\end{proof}

We will now show that when $D$ is large, the $\lambda$-probability of such a path (consisting entirely of disagreements and sites where $w_1$ or $w_2$ contains a $B$-letter) is very low. Consider any $(w_1, w_2)$ in the support of $\lambda$ containing $\gamma$ a path from $s$ to $\partial C \cup \partial C'$ contained within $\overline{C} \cap \overline{C'}$, which consists entirely of disagreements and sites where $w_1$ or $w_2$ has a $B$-letter. By passing to a subpath if necessary, we can assume that $\gamma$ is such a path of minimal length, which clearly implies that $\gamma$ is contained entirely within $C \cap C'$. Denote the length of $\gamma$ by $L$; clearly $L \geq D$. For technical reasons, we denote by $L' = 7 \lceil \frac{L}{7} \rceil \geq L$ the smallest multiple of $7$ greater than or equal to $L$ so that we can divide by $7$ in the proof without dealing with floor or ceiling functions. We quickly note a useful fact about $\gamma$: there cannot exist a site $t$ at which $w_1$ or $w_2$ contains a $B$-letter which is adjacent to three different sites on $\gamma$. Assume for a contradiction that this could happen (and that the sites adjacent to $t$ on $\gamma$ are $p,q,r$, visited in that order when $\gamma$ is traversed from $s$ to $\partial C \cup \partial C'$.) Then the path obtained by replacing the portion $p \ldots q \ldots r$ of $\gamma$ by $p t r$ would be shorter than $\gamma$, violating minimality of the length of $\gamma$. 
We need a definition:

\begin{definition} A site $t \in \gamma$ is {\bf $B$-proximate} if there is $q \in N_t \cup \{t\} \subset \gamma'$ for which $w_1(q)$ or $w_2(q)$ is a $B$-letter. 
\end{definition}

We now separate into two cases depending on whether the number of sites in $\gamma$ which are $B$-proximate is greater than or equal to $\frac{6L'}{7}$ or not. 

\

\noindent
\textbf{Case 1: $\gamma$ contains at least $\frac{6L}{7}$ $B$-proximate sites} 

It was noted earlier that any $B$-letter can be adjacent to at most two sites in $\gamma$, and so any $B$-letter can ``induce'' at most three $B$-proximate sites on $\gamma$ (up to two neighbors, and possibly itself.) We can therefore pass to a subset of $\frac{2L}{7}$ $B$-proximate sites where each is adjacent to a different $B$-letter in either $w_1$ or $w_2$, and again pass to a subset in either $w_1$ or $w_2$ (w.l.o.g. we say $w_1$) of $\lceil \frac{L}{7} \rceil = \frac{L'}{7}$ $B$-proximate sites on $\gamma$, each of which is adjacent to a different $B$-letter. 

Denote by $S$ this set of $\frac{L'}{7}$ sites on $\gamma$, and by $T$ the set of $\frac{L'}{7}$ neighboring $B$-letters in $w_1$. By Lemma~\ref{unlikelyB}, $\Lambda^{\delta}(\langle B^T \rangle) < N^{-\frac{L'}{7}}$, where $N = \lfloor \frac{1}{2} (\epsilon^{-1} - 4d - 4) \rfloor$. Since $\epsilon < \frac{1}{8d+8}$, $N > \frac{1}{5\epsilon}$. The number of possible such $T$ for any given $\gamma$ of length $L$ is bounded from above by ${L' \choose \frac{L'}{7}} (2d)^{\frac{L'}{7}} \leq (36d)^{\frac{L'}{7}}$. Therefore, the $\Lambda^{\delta}$-probability that there exists any such $T$ for a fixed $\gamma$ is bounded from above by $(180 d \epsilon)^{\frac{L'}{7}}$. Since there are fewer than $(2d)^{L'}$ possible $\gamma$ of length $L$, the $\Lambda^{\delta}$-probability that there exists any path $\gamma$ and $T$ as defined above is less than
\[
\sum_{L = D}^{\infty} \left(180 d (2d)^{7} \epsilon \right)^{\frac{L'}{7}} \leq \sum_{L = D}^{\infty} 2^{-\frac{L}{7}} = \frac{1}{1-\sqrt[7]{0.5}} 2^{-\frac{D}{7}}
\]
since $\epsilon \leq \frac{1}{360 \cdot 2^7 d^8}$. The same is true of $\Lambda^{\delta'}$, and since $\lambda$ is a coupling of $\Lambda^{\delta}$ and $\Lambda^{\delta'}$, the $\lambda$-probability that there exists any path $\gamma$ with at least $\frac{6L}{7}$ of its sites $B$-proximate is less than $\frac{2}{1-\sqrt[7]{0.5}} 2^{-\frac{D}{7}}$.

\

\noindent
\textbf{Case 2: $\gamma$ contains fewer than $\frac{6L}{7}$ $B$-proximate sites} 

In this case, there exists $R \subset \gamma$, $|R| = \lceil \frac{L}{7} \rceil = \frac{L'}{7}$, such that no site in $R$ is $B$-proximate. 
Since $\gamma$ consists entirely of sites where either one of $w_1$ and $w_2$ is a $B$-letter or $w_1$ and $w_2$ disagree, this implies that for each $r \in R$, $w_1(r) \neq w_2(r)$. Also, by the definition of $B$-proximate, for each $r \in R$, both $w_1|_{N_r}$ and $w_2|_{N_r}$ contain only $G$-letters. Order the elements of $R$ as $r_1, \ldots, r_{\frac{L'}{7}}$. Our fundamental claim is that for any $i \in [1,\frac{L'}{7}]$, 
\begin{multline}\label{unlikelydis}
\lambda \big(w_1(r_i) \neq w_2(r_i) \ | \ w_1(r_j) \neq w_2(r_j), 1 \leq j < i \\
\textrm{and } w_1|_{N_{r_{j'}}}, w_2|_{N_{r_{j'}}} \in G^{N_{r_{j'}}}, 1 \leq j' \leq i\big) < 12d\epsilon.
\end{multline}
To prove (\ref{unlikelydis}), we fix some $i \in [1,\frac{L'}{7}]$ and condition on the facts that $w_1(r_j) \neq w_2(r_j)$ for $1 \leq j < i$, and that $w_1|_{N_{r_{j'}}}, w_2|_{N_{r_{j'}}} \in G^{N_{r_{j'}}}$ for $1 \leq j' \leq i$. Then the conditional $\lambda$-distribution on $r_i$ is a weighted average of the $\lambda$-distribution assigned at site $r_i$, taken over all possible evolutions of $w_1$ and $w_2$ in the definition of $\lambda$. For any such evolution of $w_1$ and $w_2$, at the step where $w_1(r_i)$ and $w_2(r_i)$ were (simultaneously) assigned, no unassigned site in either $\overline{C}$ or $\overline{C'}$ was adjacent to an assigned $B$-letter. (Otherwise, the smallest such site lexicographically would be used instead of $r_i$, under case (i) in the definition of $\lambda$.) This means that at this step, $\underline{\partial}(\overline{C} \cap W)$ and $\underline{\partial}(\overline{C'} \cap W)$ both consist entirely of $G$-letters.


Therefore, independently of which evolution of $w_1$ and $w_2$ we consider, for any possible $\zeta_1$ when $r_i$ was assigned, Lemma~\ref{unlikelyB} implies that $\Lambda^{\zeta_1}(\langle G^{N_{r_i}} \rangle) > 1 - \frac{2d}{N} > 1 - 10d\epsilon$. This means that for any possible $\zeta_1$ when $r_i$ was assigned, $\Lambda^{\zeta_1}|_{r_i}$ was a weighted average of the conditional distributions $\Lambda^{\zeta_1}(\langle x(r_i) \rangle \ | \ \langle x|_{N_{r_i}} \rangle)$, where at least $1 - 10d\epsilon$ of the weights are associated to $x|_{N_{r_i}}$ consisting entirely of $G$-letters. For any such $x|_{N_{r_i}}$, $\Lambda^{\zeta_1}(\langle x(r_i) \rangle \ | \ \langle x|_{N_{r_i}} \rangle)$ is a uniform distribution over a subset of $A$ of size at least $|A|(1 - 2d\epsilon)$ by $\epsilon$-fullness of $X$. Therefore, for any such $x|_{N_{r_i}}$,
\[
d\left(\Lambda^{\zeta_1}(\langle x(r_i) \rangle \ | \ \langle x|_{N_{r_i}} \rangle), U\right) < 2d\epsilon,
\]
where we use $U$ to denote the uniform distribution over all of $A$. The analogous estimate also holds for $\Lambda^{\zeta_2}$ by exactly the same argument. Since at least $1 - 10d\epsilon$ of the measures $\Lambda^{\zeta_1}|_{r_i}$ and $\Lambda^{\zeta_2}|_{r_i}$ have been decomposed as weighted averages of distributions within $2d\epsilon$ of $U$, 
\[
d\left(\Lambda^{\zeta_1}|_{r_i}, \Lambda^{\zeta_2}|_{r_i}\right) < 12d\epsilon.
\]
Since the marginalization of $\lambda$ to $r_i$ is an optimal coupling of these two measures, this marginalization gives a probability of less than $12d\epsilon$ to the event $w_1(r_i) \neq w_2(r_i)$. Since the same is true for every evolution of $w_1$ and $w_2$, we have shown that conditioned on $w_1(r_j) \neq w_2(r_j)$ for $1 \leq j < i$ and $w_1|_{N_{r_{j'}}}, w_2|_{N_{r_{j'}}} \in G^{N_{r_{j'}}}$ for $1 \leq j' \leq i$, $\lambda(w_1(r_i) \neq w_2(r_i)) < 12d\epsilon$, verifying (\ref{unlikelydis}). 

From this, it is clear that $\lambda(\textrm{no site in } R \textrm{ is } B-\textrm{proximate}) < (12d\epsilon)^{\frac{L'}{7}}$ by decomposing it as a product of conditional probabilities. There are at most $(2d)^L$ choices for $\gamma$ and at most ${L' \choose \frac{L'}{7}} \leq 18^{\frac{L'}{7}}$ choices for the subset $R$, so the $\lambda$-probability that there is any path $\gamma$ with at least $\frac{L}{7}$ non-$B$-proximate sites is less than
\[
\sum_{L = D}^{\infty} (216 d (2d)^7 \epsilon)^{\frac{L'}{7}} \leq \sum_{L = D}^{\infty} 2^{-\frac{L}{7}} = \frac{1}{1-\sqrt[7]{0.5}} 2^{-\frac{D}{7}}
\]
since $\epsilon \leq \frac{1}{432 \cdot 2^7 d^8}$. 

\

Clearly any path $\gamma$ from $s$ to $\partial C \cup \partial C'$ contained within $\overline{C} \cap \overline{C'}$ which consists entirely of disagreements and locations where either $w_1$ or $w_2$ has a $B$-letter must be in either Case 1 or Case 2, so we have shown that the $\lambda$-probability that there exists any such path is less than $\frac{2}{1-\sqrt[7]{0.5}} 2^{-\frac{D}{7}} + \frac{1}{1-\sqrt[7]{0.5}} 2^{-\frac{D}{7}} < Z 2^{-\frac{D}{7}}$ for a constant $Z$ independent of $D$. By Fact~\ref{path}, $\lambda$-a.s. $w_1(s) \neq w_2(s)$ if and only if there exists such a $\gamma$, and so $\lambda(w_1(s) \neq w_2(s)) < Z 2^{-\frac{D}{7}}$. Clearly this implies via a simple union bound that for any shape $S$ consisting of sites at a distance at least $D$ from $\partial C$ and $\partial C'$, $\lambda(w_1|_S \neq w_2|_S) < Z |S| 2^{-\frac{D}{7}}$.

Since $\lambda$ is a coupling of $\Lambda^{\delta}$ and $\Lambda^{\delta'}$, we have shown the following:

\begin{fact}\label{unlikelydisfact}
For any $\delta \in L_{\partial C}(X)$ and $\delta' \in L_{\partial C'}(X)$ consisting only of $G$-letters, and for any shape $S \in C \cap C'$, 
\begin{equation}\label{unlikelydis2}
d\left(\Lambda^{\delta}|_S, \Lambda^{\delta'}|_S \right) < Z |S| 2^{-\frac{D}{7}},
\end{equation}
where $D := d(s, \partial C \cup \partial C')$.
\end{fact}

We note that (\ref{unlikelydis2}) is very close to the classical condition of (weak) spatial mixing with exponential rate (see \cite{BMP} for a survey of various results and discussions involving spatial mixing of MRFs), but with the important difference that it only holds here for boundaries consisting entirely of $G$-letters. To finish the proof, we must now consider general boundaries $\eta$. For this portion of the proof, we will use only the fact that $d\left(\Lambda^{\delta}|_S, \Lambda^{\delta'}|_S \right)$ decays to $0$ as $D \rightarrow \infty$, ignoring the exponential rate. 

Roughly speaking, the strategy is to show that for any connected set $C \subseteq \mathbb{Z}^d$, any measure of maximal entropy $\mu$ on $X$ and for any finite shape $S \subset C$ far from $\partial C$, there are sets of boundary conditions $\eta \in L_{\partial C}(X)$ with $\mu$-measure approaching $1$ (as $d(S,\partial C) \rightarrow \infty$) whose members have the following property: with very high $\Lambda^{\eta}$-probability there exists a closed contour $\delta$ of $G$-letters contained in $C$, containing $S$, and which is far from $S$. Then, for any such $\eta$, most of $\Lambda^{\eta}|_S$ can be written as a weighted average over $\Lambda^{\delta}|_S$ for such $\delta$. We have already shown that $\Lambda^{\delta}|_S$ has very little dependence on $\delta$ consisting only of $G$-letters when $\delta$ is far from $S$, and so $\Lambda^{\eta}(\langle x|_S \rangle)$ has little dependence on $\eta$. We can write $\mu|_S$ as a weighted average of $\Lambda^{\eta}|_S$, and since the above shows that dependence on $\eta$ fades as $C$ becomes large for sets of $\eta$ of measure approaching $1$, $\mu|_S$ has only one possible value. Since $S$ was arbitrary, this shows that $\mu$ is the unique measure of maximal entropy on $X$. 

In the sequel, we use the notation $S \leftrightarrow T$ to denote the event that there is a path of $B$-letters connecting some site in $S$ to some site in $T$. We first need to prove the following:

\begin{fact}\label{Bpathdecayfact}
For any measure of maximal entropy $\mu$ on $X$, 
\begin{equation}\label{Bpathdecay}
\lim_{n \rightarrow \infty} \mu(0 \leftrightarrow \partial [-n,n]^d) = 0.
\end{equation}
\end{fact}

\begin{proof}
For a contradiction, suppose (\ref{Bpathdecay}) is false. Then since the events $0 \leftrightarrow \partial [-n,n]^d$ are decreasing, there exists $\alpha > 0$ so that for all $n$, $\mu(0 \leftrightarrow \partial [-2n,2n]^d) > \alpha$. Then by stationarity, for each $s \in [-n,n]^d$, $\mu(s \leftrightarrow (s + \partial [-2n,2n]^d)) > \alpha$. Since $s + [-2n,2n]^d \supset [-n,n]^d$, $s \leftrightarrow (s + \partial [-2n,2n]^d)$ implies $s \leftrightarrow \partial [-n,n]^d$, and so $\mu(s \leftrightarrow \partial [-n,n]^d) > \alpha$ for all $s \in [-n,n]^d$. Then 
\begin{equation}\label{chebyshev1}
\mu\left(|\{s \in [-n,n]^d \ : \ s \leftrightarrow \partial [-n,n]^d\}| > 0.5\alpha |[-n,n]^d|\right) > 0.5\alpha.
\end{equation}
Since $\mu$ is an MRF with conditional probabilities $\{\Lambda^{\delta}\}$, we may write $\mu|_{[-n,n]^d}$ as a weighted average:
\[
\mu|_{[-n,n]^d} = \sum \rho_i \Lambda^{\delta_i},
\]
where $\delta_i$ ranges over configurations in $L_{\partial [-n,n]^d}(X)$ (and $\rho_i = \mu([\delta_i])$.) By (\ref{chebyshev1}), at least $0.25\alpha$ of the weights $\rho_i$ are associated to $\delta_i$ for which 
\begin{equation}\label{badborder}
\Lambda^{\delta_i}(|\{s \in [-n,n]^d \ : \ s \leftrightarrow \partial [-n,n]^d\}| > 0.5\alpha |[-n,n]^d|) > 0.25\alpha.
\end{equation}

In other words, if we denote the set of $\delta_i$ which satisfy (\ref{badborder}) by $P$, then $\mu([P]) > 0.25\alpha$. Make the notation $K = |\{w \in L_{[-n-1,n+1]^d}(X) \ : \ w|_{\partial [-n,n]^d} \in P\}|$. Recall that $- \sum_{w \in L_{[-n-1,n+1]^d}(X)} \mu([w]) \log \mu([w]) \geq |[-n-1,n+1]^d| h(\mu)$, which equals $|[-n-1,n+1]^d| h(X)$ since $\mu$ is a measure of maximal entropy. Then 
\begin{multline}\label{entdecomp}
|[-n-1,n+1]^d| h(X) \leq - \sum_{w \in L_{[-n-1,n+1]^d}(X)} \mu([w]) \log \mu([w]) = \\
- \sum_{\delta \in P^c} \sum_{w \in L_{[-n-1,n+1]^d}(X), w|_{\partial [-n,n]^d} = \delta} \mu([w]) \log \mu([w])\\
- \sum_{\delta \in P} \sum_{w \in L_{[-n-1,n+1]^d}(X), w|_{\partial [-n,n]^d} = \delta} \mu([w]) \log \mu([w])\\
\leq (1 - 0.25\alpha) \log |L_{[-n-1,n+1]^d}(X)| + 0.25\alpha \log K - \log (0.25 \alpha),
\end{multline}
where the last inequality uses the easily checked fact that for any positive real numbers $\beta_1, \ldots, \beta_k\}$ with sum $\beta$, $\sum (-\beta_i \log \beta_i) \leq \beta (\log k - \log \beta)$.

By definition of topological entropy, for any $\theta > 0$, there exists $N_{\theta}$ such that \newline $(1 - 0.25 \theta \alpha) \log |L_{[-n,n]^d}(X)| < |[-n,n]^d| h(X) + \log(0.25\alpha)$ for $n > N_{\theta}$. This means, in particular, that for $n \geq N_{\theta}$, 
\begin{multline*}
(1 - 0.25 \theta \alpha) \log |L_{[-n-1,n+1]^d}(X)| < |[-n-1,n+1]^d| h(X) + \log (0.25\alpha)\\
\leq (1-0.25\alpha) \log |L_{[-n-1,n+1]^d}(X)| + 0.25\alpha \log K \quad {\rm (by \ (\ref{entdecomp}))}\\
\Longrightarrow \log K > (1 - \theta) \log |L_{[-n-1,n+1]^d}(X)|.
\end{multline*}

Therefore, for $n > N_{\theta}$, $K > |L_{[-n-1,n+1]^d}(X)|^{1 - \theta}$. Since there are fewer than $|A|^{|\partial [-n,n]^d|}$ elements of $P$, there exists $\delta \in P$ and a set of at least $\frac{|L_{[-n-1,n+1]^d}(X)|^{1 - \theta}}{|A|^{|\partial [-n,n]^d|}}$ configurations $w$ for which $w\delta \in L(X)$. Then, since $\delta$ satisfies (\ref{badborder}), there is a set of configurations $S \subseteq L_{[-n-1,n+1]^d}(X)$ of size at least $\frac{0.25\alpha |L_{[-n-1,n+1]^d}(X)|^{1 - \theta}}{|A|^{|\partial [-n,n]^d|}}$, each of which contains at least $0.5\alpha |[-n,n]^d|$ sites connected to $\partial [-n,n]^d$ by paths of $B$-letters.

We now perform a very similar replacement procedure to the one used in the proof of Lemma~\ref{unlikelyB}. We will not rewrite the entire construction, rather mainly summarizing the changes from the previous procedure. Consider any $u \in S$, and take $C_i(u)$, $1 \leq i \leq k$, to be the maximal connected components of locations of $B$-letters in $u$ which have nonempty intersection with $\partial [-n,n]^d$. Since $u \in S$, $\sum |C_i(u)| > 0.5\alpha |[-n,n]^d|$. For each $i$, define $B_i(u)$ to be the subconfiguration $u|_{C_i(u)}$ of $u$ occupying $C_i(u)$. Then, remove all $B_i(u)$ from $u$, and fill the holes in various ways using the same procedure as in Lemma~\ref{unlikelyB}, where each site $s$ is filled with a $G$-letter which encodes the information about the $B$-letter $u(s)$ and whether $s$ was on the boundary of its component $C_i(u)$ or in the interior. As in Lemma~\ref{unlikelyB}, this yields a set $f(u) \subseteq L_{[-n-1,n+1]^d}(X)$ of configurations of size at least $N^{0.5\alpha |[-n,n]^d|}$, where $N = \lfloor \frac{1}{2} (\epsilon^{-1} - 4d - 4) \rfloor$. Here, we will actually only need the fact that $N > 2$, which is true since $\epsilon < \frac{1}{4d + 8}$. We have
\[
\sum_u |f(u)| > |S| 2^{0.5\alpha |[-n,n]^d|}.
\]

In Lemma~\ref{unlikelyB}, we showed that all of the sets $f(u)$ were disjoint, which is not necessarily the case here. However, it is still true that if there exist $C_i(u)$ and $C_j(u')$ which are unequal but have nonempty intersection, then $f(u) \cap f(u') = \varnothing$. It is also still true that if there exist $C_i(u)$ and $C_j(u')$ which are equal, but $B_i(u) \neq B_j(u')$, then $f(u) \cap f(u') = \varnothing$. The only new case under which $f(u)$ and $f(u')$ might not be disjoint is if all pairs $C_i(u)$ and $C_j(u')$ are either disjoint or equal, and if $B_i(u) = B_i(u')$ whenever $C_i(u) = C_j(u')$; suppose we are in this case. Fix any $v \in L_{[-n-1,n+1]^d}(X)$, and let us bound the size of $F_v := \{u \ : \ v \in f(u)\}|$ from above.

For each $s \in \partial [-n,n]^d$, either $s$ is in some $C_i(u_s)$ for some configuration $u_s \in F_v$ or not. If it is, then denote the configuration $B_i(u_s)$ by $B(s)$. By the above analysis, for every $u \in F_v$, either $B(s) = B_i(u)$ for some $i$ or all $B_i(u)$ are disjoint from $B(s)$ (in particular, this would imply that no $B_i(u)$ contains $s$). This in turn implies that for every $u \in F_v$, the set $\{B_i(u)\}$ is just a subset of $\{B(s)\}_{s \in \partial [-n,n]^d}$. Since knowing $\{B_i(u)\}$ along with $v$ uniquely determines $u$, and since there are at most $|\partial[-n,n]^d|$ sets $B(s)$, $|F_v| \leq 2^{|\partial [-n,n]^d|}$. In other words, each $v$ is in at most $2^{|\partial [-n,n]^d|}$ of the sets $f(u)$. Since this is true for any $v$, we have shown that
\begin{multline*}
|L_{[-n-1,n+1]^d}(X)| \geq \frac{\sum_u |f(u)|}{2^{|\partial [-n,n]^d|}} \geq \frac{|S| 2^{0.5\alpha |[-n-1,n+1]^d|}}{2^{|\partial [-n,n]^d|}} \\
\geq |L_{[-n-1,n+1]^d}(X)|^{1 - \theta} \frac{0.25\alpha 2^{0.5\alpha |[-n,n]^d|}}{(2|A|)^{|\partial [-n,n]^d|}} \\
\Longrightarrow |L_{[-n-1,n+1]^d}(X)| \geq \left(\frac{0.25\alpha 2^{0.5\alpha |[-n,n]^d|}}{(2|A|)^{|\partial [-n,n]^d|}}\right)^{\theta^{-1}}.
\end{multline*}
However, since $|L_{[-n-1,n+1]^d}(X)| < |A|^{|[-n-1,n+1]^d|}$, this clearly gives a contradiction for small enough $\theta$ and sufficiently large $n$ (both larger than $N_{\theta}$ and large enough so that $\frac{|\partial [-n,n]^d|}{|[-n,n]^d|}$ is much smaller than $0.5\alpha$). Therefore, our original assumption was wrong and (\ref{Bpathdecay}) is true, i.e. $\lim_{n \rightarrow \infty} \mu(0 \leftrightarrow \partial [-n,n]^d) = 0$.

\end{proof}

We are now ready to complete the proof of (A). Choose $\mu$ to be any measure of maximal entropy on $X$, and fix any $k$, $l$, and $\epsilon > 0$. By Fact~\ref{Bpathdecayfact}, we can choose $n > k+l$ large enough that $\mu(0 \leftrightarrow \partial [-n+k+l,n-k-l]^d) < \frac{\epsilon}{|[-k-l,k+l]^d|}$. Then by stationarity of $\mu$, $\mu(t \leftrightarrow t + \partial [-n+k+l,n-k-l]^d) < \frac{\epsilon}{|[-k-l,k+l]^d|}$ for all $t \in [-k-l,k+l]^d$. Since $[-n,n]^d \supseteq t + [-n+k+l,n-k-l]^d$ for all such $t$, the event $t \leftrightarrow \partial [-n,n]^d$ is contained in the event $t \leftrightarrow t + \partial [-n+k+l,n-k-l]^d$ for all such $t$, and so $\mu(t \leftrightarrow \partial [-n,n]^d) < \frac{\epsilon}{|[-k-l,k+l]^d|}$. Summing over all $t \in [-k-l,k+l]^d$ yields $\mu([-k-l,k+l]^d \leftrightarrow \partial [-n,n]^d) < \epsilon$.
 
This implies that for any $n$, there is a set $U_n \subseteq L_{[-n,n]^d}(X)$ with $\mu([U_n]) > 1 - \epsilon$ such that any $w \in U_n$ contains a closed contour consisting entirely of $G$-letters containing $[-k-l,k+l]^d$ in its interior. For any $w \in U_n$, if $\{\gamma_i\} = \{\partial S_i\}$ is the collection of all such contours, then clearly $\gamma(w) := \partial (\bigcup S_i)$ is the unique maximal such contour, i.e. any other such closed contour $\gamma'$ for $w$ is contained in the union of $\gamma(w)$ and its interior. Define $B(w)$ to consist of the set of all sites of $[-n,n]^d$ on or outside $\gamma(w)$, and $D(w) = w|_{B(w)}$. We note that $[U_n]$ can be written as a disjoint union of the sets $[L_{[-n,n]^d}(X) \cap \langle D(w) \rangle]$ over all possible choices for $D(w)$. (For clarity, we note that $L_{[-n,n]^d}(X) \cap \langle D(w) \rangle$ consists of all configurations $x$ in $L_{[-n,n]^d}(X)$ for which $x|_{B(w)} = D(w)$.) This means that $\mu$, restricted to $U_n$ and then marginalized to $[-k,k]^d$, can be written as a weighted average of the conditional measures $\mu\left([ x|_{[-k,k]^d} ] \ | \ [D(w)]\right)$ over possible values of $D(w)$, and since $\mu$ is an MRF, this is actually a weighted average of $\Lambda^{\gamma(w)}|_{[-k,k]^d}$.


However, each $\gamma(w)$ is a closed contour of $G$-letters with distance greater than $l$ from $[-k,k]^d$, and so by Fact~\ref{unlikelydisfact}, for any $\gamma(w)$ and any $\eta \in L_{\partial [-n,n]^d}(X)$ consisting only of $G$-letters,
\begin{equation}\label{closegammas}
d\left(\Lambda^{\gamma(w)}|_{[-k,k]^d}, {\eta}|_{[-k,k]^d}\right) < Z |[-k,k]^d| 2^{-\frac{l}{7}}.
\end{equation}
Since the set $U_n$ has $\mu$-measure at least $1 - \epsilon$, and since $\mu|_{[-k,k]^d}$ restricted to $U_n$ can be decomposed as a weighted average of measures $\Lambda^{\gamma(w)}|_{[-k,k]^d}$, (\ref{closegammas}) implies that
\begin{equation}\label{closegammas2}
d\left(\mu|_{[-k,k]^d}, \Lambda^{\eta}|_{[-k,k]^d}\right) < Z |[-k,k]^d| 2^{-\frac{l}{7}} + \epsilon.
\end{equation}
By taking $l \rightarrow \infty$ and $\epsilon \rightarrow 0$ (thus forcing $n \rightarrow \infty$, since $n$ was chosen larger than $k+l$), we see that $\mu|_{[-k,k]^d}$ is in fact uniquely determined by the conditional probabilities $\Lambda^{\delta}$. Since $k$ was arbitrary, $\mu$ is the unique shift-invariant MRF with conditional probabilities $\Lambda^{\delta}$, implying by Proposition~\ref{P1} that $\mu$ is the unique measure of maximal entropy on $X$, proving (A).

\end{proof}

We now state two corollaries of the proof of (A), which will be useful later for the proofs of (B) - (D).

\begin{corollary}\label{glimit}
If $X$ is an $\epsilon_d$-full nearest neighbor $\mathbb{Z}^d$ SFT with unique measure of maximal entropy $\mu$, then $\mu$ is the (unique) weak limit (as $n \rightarrow \infty$) of $\Lambda^{\eta_n}$ for any sequence $\eta_n \in L_{\partial [-n,n]^d}(X)$ of boundary configurations consisting only of $G$-letters.
\end{corollary}

\begin{proof}
Choose any such sequence $\eta_n$. For any $k$, (\ref{closegammas2}) implies that as $n \rightarrow \infty$, $\Lambda^{\eta_n}|_{[-k,k]^d}$ approaches $\mu|_{[-k,k]^d}$ weakly. Therefore, $\Lambda^{\eta_n} \rightarrow \mu$.
\end{proof}

\begin{corollary}\label{gwords}
If $X$ is an $\epsilon_d$-full nearest neighbor $\mathbb{Z}^d$ SFT with unique measure of maximal entropy $\mu$, then any configuration $u \in L(X)$ containing only $G$-letters on its inner boundary has positive $\mu$-measure.
\end{corollary}

\begin{proof}
Consider any such configuration $u \in L_S(X)$. It was shown in the proof of (A) that there exists $T \supset S$ and a closed contour $\delta \in L_{\partial T}$ of $G$-letters containing $S$ for which $\mu([\delta])>0$. (Specifically, take $k,l,n$ large enough that $S \subseteq [-k-l,k+l]^d$ and $\mu([U_n]) > 0$, choose $w \in U_n$ with $\mu([w]) > 0$, and then take $\delta = \gamma(w)$.) Since the concatenation $u \delta$ has inner boundary consisting only of $G$-letters, by Corollary~\ref{LA=GA} it is globally admissible. Therefore, there exists a configuration $v \in L(X)$ with $v|_S = u$ and $v|_{\partial T} = \delta$, implying that $\Lambda^{\delta}(\langle u \rangle) > 0$. Then $\mu([u]) \geq \mu([u \delta]) = \mu([\delta]) \mu([u] \ | \ [\delta]) = \mu([\delta]) \Lambda^{\delta}(\langle u \rangle) > 0$.
\end{proof}

\begin{remark}\label{notimprovement}
At first glance, (A) may appear to be an extension of the main result from \cite{vdBM}, which guaranteed uniqueness of MRFs corresponding to certain classes of conditional probabilities. However, this is not the case; even for $\epsilon$ arbitrarily close to $0$, $\epsilon$-full SFTs may still support multiple MRFs with the same uniform conditional probabilities $\Lambda^{\delta}$, some corresponding to limits of boundary conditions involving $B$-letters. (For instance, in Example~\ref{onepoint}, both the point mass at $0^{\mathbb{Z}^d}$ and the Bernoulli measure of maximal entropy on $\{1,\ldots,n\}^{\mathbb{Z}^d}$ are MRFs with conditional probabilities $\Lambda^{\delta}$.) We very much need the extra condition of maximal entropy to rule out all but one of these MRFs as ``degenerate.''
\end{remark}

\

\begin{proof}[Proof of (B)] Our strategy is to first show that we can compute $h(X)$ by taking the exponential growth rate of globally admissible configurations whose boundaries contain only $G$-letters, and that we can bound the rate at which these approximations approach $h(X)$. Then, we can easily write an algorithm which counts such configurations, since by Corollary~\ref{LA=GA}, a configuration with boundary containing only $G$-letters is globally admissible iff it is locally admissible.

Fix any $n$, and denote by $\Gamma$ the set $\underline{\partial} [1,n]^d$. For any $\delta \in L_{\Gamma}(X)$, we will show that $|L_{[1,n]^d}(X) \cap \langle\delta\rangle| \leq |L_{[1,n]^d}(X) \cap \langle G^{\Gamma} \rangle|$, i.e. the number of globally admissible configurations with shape $[1,n]^d$ whose restriction to $\Gamma$ equals $\delta$ is less than or equal to the number of globally admissible configurations with shape $[1,n]^d$ whose restriction to $\Gamma$ consists entirely of $G$-letters. The proof involves another replacement procedure similar to the one from the proof of Lemma~\ref{unlikelyB}; the difference is that for any $u \in L_{[1,n]^d}(X) \cap \langle \delta\rangle$, we will define only a single configuration $f(u)$ in $L_{[1,n]^d}(X)$, rather than a set. Take $C_i(u)$, $1 \leq i \leq k(u)$, to be the set of maximal connected components of locations of $B$-letters in $u$ which have nonempty intersection with $\Gamma$, and for each $i$ define $B_i(u) = u|_{C_i(u)}$ to be the subconfiguration of $u$ occupying $C_i(u)$. Just as in the proof of Lemma~\ref{unlikelyB}, remove all $B_i(u)$ from $u$, and fill the holes with configurations of $G$-letters, where the letter chosen to fill a site $s$ encodes the letter $u(s)$ and the information of whether $s$ was on the boundary of its component $C_i(u)$ or in the interior. For exactly the same reasons as in Lemma~\ref{unlikelyB}, $u \neq u' \Rightarrow f(u) \neq f(u')$. We also note that all $f(u)$ are in $L_{[1,n]^d}(X) \cap \langle G^{\Gamma} \rangle$, meaning that their restrictions to $\Gamma$ consist entirely of $G$-letters.

We have then shown that $|L_{[1,n]^d}(X) \cap \langle \delta\rangle | \leq |L_{[1,n]^d}(X) \cap \langle G^{\Gamma}\rangle |$. By summing over all choices for $\delta$, we see that $|L_{[1,n]^d}(X)| \leq |A|^{|\Gamma|} |L_{[1,n]^d}(X) \cap \langle G^{\Gamma}\rangle |$. This means that 
\begin{multline}\label{upperbound}
h(X) \leq \frac{1}{n^d} \log |L_{[1,n]^d}(X)| \leq \frac{1}{n^d} \left(\log |L_{[1,n]^d}(X) \cap \langle G^{\Gamma}\rangle | + |\Gamma| \log |A|\right) \\ \leq \frac{1}{n^d} \log |L_{[1,n]^d}(X) \cap \langle G^{\Gamma}\rangle | + \frac{2d \log |A|}{n}.
\end{multline}

We now make a simple observation: for any $k$ and any configurations $w_t \in L_{[1,n]^d}(X) \cap \langle G^{\Gamma}\rangle $, $t \in [1,k]^d$, define the concatenation $u$ of all $w_t$, which has shape $\bigcup_{t} \prod [1+(t_i-1)(n+1),t_i(n+1) - 1]$. 
Then $u$ is made up of a union of locally admissible configurations where each pair is separated by a distance of at least $1$, so it is locally admissible. Then by Corollary~\ref{LA=GA}, since the outer boundary of $u$ consists only of $G$-letters and $\epsilon < \frac{1}{2d+2}$, $u$ is also globally admissible, meaning in particular that it is a subconfiguration of a configuration in $L_{[1,k(n+1)]^d}(X)$. This implies that for any $k > 0$,
\[
|L_{[1,k(n+1)]^d}(X)| \geq |L_{[1,n]^d}(X) \cap \langle G^{\Gamma}\rangle |^{k^d}.
\]
By taking logs of each side, dividing by $(k(n+1))^d$, and letting $k \rightarrow \infty$, we see that
\begin{equation}\label{lowerbound}
h(X) \geq \frac{1}{(n+1)^d} \left(\log |L_{[1,n]^d}(X) \cap \langle G^{\Gamma}\rangle |\right).
\end{equation}
The upper and lower bounds on $h(X)$ given by (\ref{upperbound}) and (\ref{lowerbound}) differ by
\begin{multline}
\frac{2d \log |A|}{n} + \left(\frac{1}{n^d} - \frac{1}{(n+1)^d}\right) \log \left(|L_{[1,n]^d}(X) \cap \langle G^{\Gamma}\rangle |\right) \leq \frac{2d \log |A|}{n} \\
+ \frac{(n+1)^d - n^d}{n^d (n+1)^d} n^d \log |A| \leq \frac{2d \log |A|}{n} + \frac{d(n+1)^{d-1}}{(n+1)^d} \log |A| \leq \frac{3d \log |A|}{n}.
\end{multline}
Since $\frac{1}{n^d} \log |L_{[1,n]^d}(X) \cap \langle G^{\Gamma}\rangle |$ is between the bounds from (\ref{upperbound}) and (\ref{lowerbound}), it is within $\frac{3d \log |A|}{n}$ of $h(X)$. By Corollary~\ref{LA=GA}, $\log |L_{[1,n]^d}(X) \cap \langle G^{\Gamma}\rangle | = \newline \log |{[1,n]^d}(X) \cap \langle G^{\Gamma}\rangle |$. Finally, we note that $\log \left(|LA_{[1,n]^d}(X) \cap \langle G^{\Gamma}\rangle |\right)$ can be computed algorithmically, in $|A|^{n^d(1 + o(1))}$ steps, by simply writing down all possible configurations with alphabet $A$ and shape $[1,n]^d$ and counting those which are locally admissible and have restriction to $\Gamma$ consisting only of $G$-letters. 

Since we may invest $|A|^{n^d(1 + o(1))}$ steps to get an approximation to $h(X)$ with tolerance $\frac{3d \log |A|}{n}$, clearly $h(X)$ is computable in time $e^{O(n^d)}$, verifying (B).

\end{proof}

\begin{proof}[Proof of (C)] For any $n$, we define $X_n$ to be the $\mathbb{Z}^d$ SFT consisting of all points of $X$ in which all connected components of $B$-letters have size less than $n$. We will show that each $X_n$ has the UFP and that $h(X_n) \rightarrow h(X)$ as $n \rightarrow \infty$. 

We first verify that $X_n$ has the UFP with distance $2n$. Consider any $k,l$ with $l > k+3n$ and any configurations $w \in L_{[-k,k]^d}(X_n)$ and $w' \in L_{[-l,l]^d \setminus [-k-2n,k+2n]^d}(X_n)$. We will exhibit $x \in X_n$ with $x|_{[-k,k]^d} = w$ and $x|_{[-l,l]^d \setminus [-k-2n,k+2n]^d} = w'$, proving the UFP once one takes weak limits with $l \rightarrow \infty$.

We first use the fact that $w, w'$ are globally admissible in $X_n$ to extend them to configurations $v \in L_{[-k-n+1,k+n-1]^d}$ and $v' \in L_{[-l,l]^d \setminus [-k-n-1,k+n+1]^d}$ respectively. Then, in both $v$ and $v'$, remove any connected components of $B$-letters which have empty intersection with $w$ or $w'$. Fill these with $G$-letters in some locally admissible way by Lemma~\ref{manyfillings}, creating configurations $u$ and $u'$ respectively. Since connected components of $B$-letters in $X_n$ must have size less than $n$, $u|_{\underline{\partial} [-k-n+1,k+n-1]^d}$ and $u'|_{\underline{\partial} [-k-n-1,k+n+1]^d}$ consist only of $G$-letters. (If this were not the case, then either $u$ contained a connected component of $B$-letters intersecting both $[-k,k]^d$ and $\underline{\partial} [-k-n+1,k+n-1]^d$ or $u'$ contained a connected component of $B$-letters intersecting both $[-l,l]^d \setminus [-k-2n,k+2n]^d$ and $\underline{\partial} [-k-n-1, k+n+1]^d$, and in either case such a component would have had size at least $n$, which is impossible since $v,v' \in L(X_n)$.) Then again by Corollary~\ref{manyfillings}, the empty region $\underline{\partial} [-k-n,k+n]^d$ between $u$ and $u'$ can be filled with $G$-letters in a locally admissible way, creating a new locally admissible configuration $v''$ with shape $[-l,l]^d$. Finally, we note that since $w'$ was globally admissible, there exists $x' \in X$ with $x'|_{[-l,l]^d \setminus [-k-2n,k+2n]^d} = w'$. Finally, we note that $w'$ has ``thickness'' at least $n$, and that no letters on $w'$ were changed in the construction of $v''$. Therefore, since $X_n$ is an SFT defined by forbidden configurations of size at most $n$, the point $x \in A^{\mathbb{Z}^d}$ defined by $x|_{[-l,l]^d} = v''$ and $x|_{\mathbb{Z}^d \setminus [-l,l]^d} = x'|_{\mathbb{Z}^d \setminus [-l,l]^d}$ is in $X_n$, and we have proved that $X_n$ has the UFP with distance $2n$. (See Figure~\ref{UFPproof} for an illustration of the creation of $x$.)

\begin{figure}[h]
\centering
\includegraphics[scale=0.3]{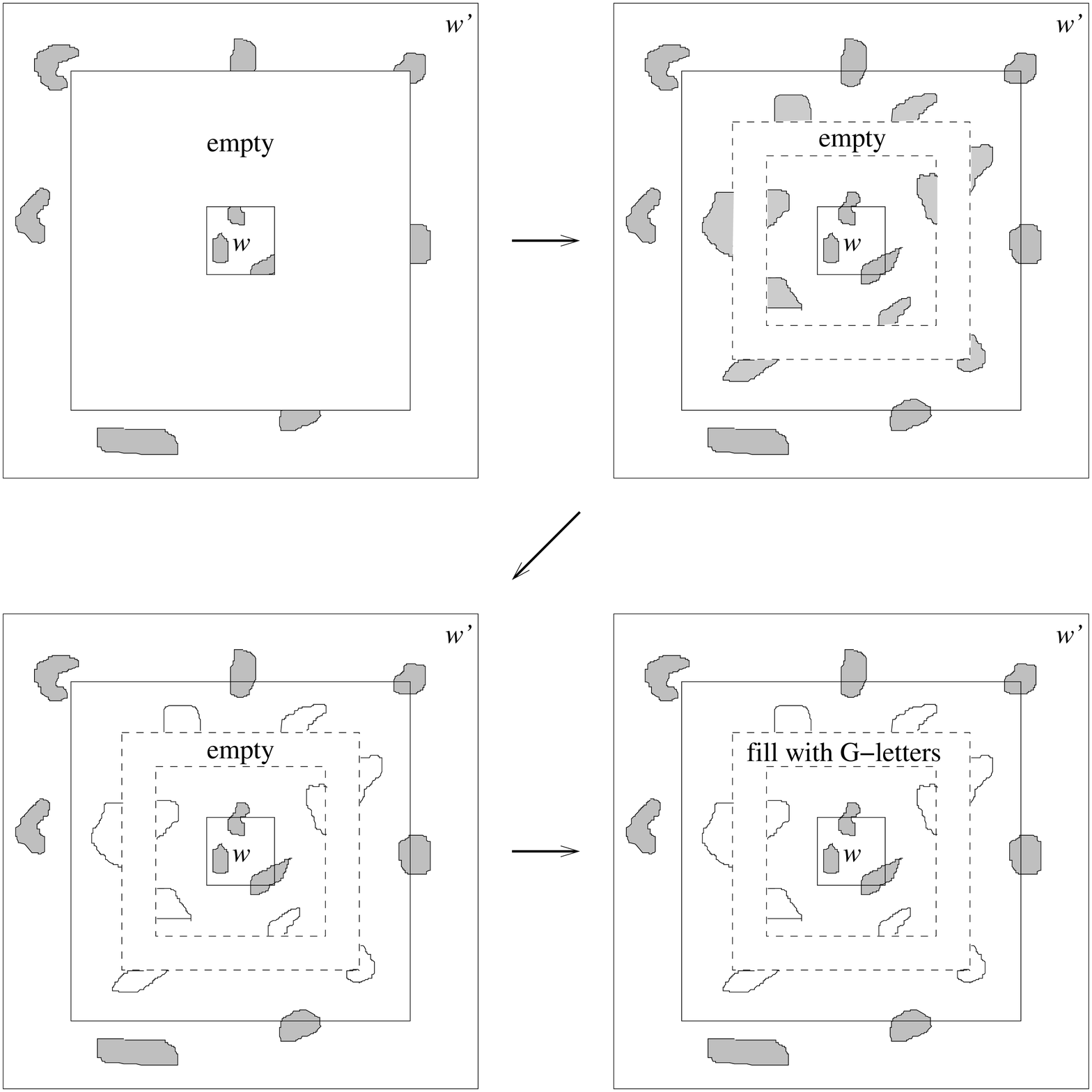}
\caption{Filling between $w$ and $w'$ (shaded areas represent $B$-letters, white areas represent $G$-letters)}
\label{UFPproof}
\end{figure}

We finish by verifying that $h(X_n) \rightarrow h(X)$. We showed in the proof of (B) that for any collection $w_t \in L_{[1,n]^d}(X) \cap \langle G^{\underline{\partial} [1,n]^d}\rangle$, $t \in [1,k]^d$, the concatenation $u$ of all $w_t$, which has shape $\bigcup_{t} \prod [1+(t_i-1)(n+1),t_i(n+1) - 1]$, is in $L(X)$.

To prove this fact, we invoked Corollary~\ref{LA=GA}, which in fact says a bit more; it implies that $u$ can be extended to a point $x \in X$ by appending only $G$-letters to $u$. Note that since each $w_t$ contains only $n^d$ sites, and since all letters of $x$ outside $u$ are in $G$, $x$ does not contain any connected components of $B$-letters with size more than $n^d$. Therefore, $x \in X_{n^d}$, which implies that $u \in L(X_{n^d})$. By counting the possible choices for the collection $(w_t)$, we see that
\[
|L_{[1,k(n+1)]^d}(X_{n^d})| \geq |L_{[1,n]^d}(X) \cap \langle G^{\underline{\partial} [1,n]^d}\rangle |^{k^d}.
\]
By taking logs of both sides, dividing by $(k(n+1))^d$, and letting $k \rightarrow \infty$, we see that
\[
h(X_{n^d}) \geq \frac{1}{(n+1)^d} \log |L_{[1,n]^d}(X) \cap \langle G^{\underline{\partial} [1,n]^d}\rangle |.
\]
We now recall that in the proof of (B), we showed that $\frac{1}{(n+1)^d} \log |L_{[1,n]^d}(X) \cap \langle G^{\underline{\partial} [1,n]^d}\rangle |$ is within $\frac{3d\log|A|}{n}$ of $h(X)$, and so we have shown that $h(X_{n^d}) \geq h(X) - \frac{3d\log|A|}{n}$ for all $n$, implying that $h(X_n) \rightarrow h(X)$ as $n \rightarrow \infty$.

For any $\mathbb{Z}^d$ aperiodic ergodic measure-theoretic dynamical system $(Y,\mu,S_t)$ with $h(\mu) < h(X)$, there then exists $n$ for which $h(\mu) < h(X_n)$. We recall from Section~\ref{defns} that any $\mathbb{Z}^d$ SFT with the UFP is a measure-theoretic universal model, and so there exists a measure $\nu$ on $X_n$ so that $(X_n, \nu, \sigma_t) \cong (Y, \mu, S_t)$. Since the support of $\nu$ is contained in $X_n$, clearly it is contained in $X$ as well, and we have verified (C).

\end{proof}

\begin{proof}[Proof of (D)] We prove that $\mu$ is isomorphic to a Bernoulli measure by using the property of quite weak Bernoulli as defined in \cite{BS2}.

\begin{definition}
A measure $\mu$ on $A^{\mathbb{Z}^d}$ is called \textbf{quite weak Bernoulli} if for all $\epsilon > 0$, 
\[
\lim_{n \rightarrow \infty} d\left(\mu|_{[-n(1-\epsilon),n(1-\epsilon)]^d \cup \mathbb{Z}^d \setminus [-n,n]^d}, \mu|_{[-n(1-\epsilon),n(1-\epsilon)]^d} \times \mu|_{\mathbb{Z}^d \setminus [-n,n]^d}\right) = 0.
\]
\end{definition}
It is known that quite weak Bernoulli measures are isomorphic to Bernoulli measures (for instance, in \cite{BS2}, they note that it implies the property of very weak Bernoulli as defined in \cite{kamm}, and that Theorem 1.1 of \cite{kamm} shows that all very weak Bernoulli measures are isomorphic to Bernoulli measures), and so it suffices to show that $\mu$ is quite weak Bernoulli. 

It is shown in \cite{BS2} that $\mu$ is quite weak Bernoulli if and only if for all $\epsilon > 0$,
\begin{multline}
\lim_{n \rightarrow \infty} \min\Big\{\alpha \ : \ \mu\big(\{\eta \in L_{\mathbb{Z}^d \setminus [-n,n]^d}(X) \ : \ \\ d\left(\mu|_{[-n(1-\epsilon),n(1-\epsilon)]^d}, \mu^{\eta}|_{[-n(1-\epsilon),n(1-\epsilon)]^d}\right) < \alpha\} \big) > 1 - \alpha \Big\} = 0,
\end{multline}
where $\mu^{\eta}$ is the conditional distribution on $[-n,n]^d$ of $\mu$ given $\eta$. Since $\mu$ is an MRF with uniform conditional probabilities $\Lambda^{\delta}$, we can replace $\mu^{\eta}$ by $\Lambda^{\delta}$, where $\delta := \eta|_{\partial [-n,n]^d}$. Therefore, it suffices to show that for all $\epsilon > 0$,
\begin{multline}\label{QWB}
\lim_{n \rightarrow \infty} \min\Big\{\alpha \ : \ \mu\big(\{\delta \in L_{\partial [-n,n]^d}(X) \ : \ \\ d\left(\mu|_{[-n(1-\epsilon),n(1-\epsilon)]^d}, \Lambda^{\delta}|_{[-n(1-\epsilon),n(1-\epsilon)]^d}\right) < \alpha\} \big) > 1 - \alpha \Big\} = 0.
\end{multline}

We first note that combining Lemma~\ref{unlikelyB} with Corollary~\ref{glimit} yields the fact that for any finite $T \subseteq \mathbb{Z}^d$, $\mu([B^T]) \leq N^{-|T|}$, where $N = \lfloor \frac{1}{2} (\epsilon^{-1} - 4d - 4) \rfloor > 4d$ since $\epsilon < \frac{1}{12d+4}$. By summing over all possible paths of $B$-letters from $\partial [-n(1-\epsilon),n(1-\epsilon)]^d$ to $\underline{\partial} [-n,n]^d$, this implies that
\[
\mu\left([-n(1-\epsilon),n(1-\epsilon)]^d \leftrightarrow \underline{\partial} [-n,n]^d\right) \leq \sum_{L = n\epsilon}^{\infty} (2d)^L N^{-L} = \frac{2d}{N-2d} \left(\frac{2d}{N}\right)^{n\epsilon}.
\]
Since $N>4d$, $\mu\left([-n(1-\epsilon),n(1-\epsilon)]^d \leftrightarrow \underline{\partial} [-n,n]^d\right) < 2^{-\epsilon n}$. Therefore, with $\mu$-probability at least $1 - 2^{-\epsilon n}$ there exists a closed contour of $G$-letters containing $[-n(1-\epsilon),n(1-\epsilon)]^d$ in its interior and contained within $[-n,n]^d$. Since $\mu$ is an MRF with uniform conditional probabilities $\Lambda^{\delta}$, $\mu|_{[-n,n]^d} = \sum_{\delta} \mu([\delta]) \Lambda^{\delta}$, where the sum is over $\delta \in L_{\partial [-n,n]^d}(X)$.  Clearly, there then exists a set 
$\Delta \subseteq L_{\partial [-n,n]^d}(X)$ with $\mu([\Delta]) > 1 - 2^{-0.5\epsilon n}$ so that for any $\delta \in \Delta$, the $\Lambda^{\delta}$-probability that there exists a closed contour of $G$-letters containing $[-n(1-\epsilon),n(1-\epsilon)]^d$ and contained in $[-n,n]^d$ is at least $1 - 2^{-0.5\epsilon n}$.

As in the proof of (A), for any $\delta \in \Delta$ and any $u \in L_{[-n,n]^d}(X)$ such that $u \delta \in L(X)$, we define $\gamma(u)$ to be the unique maximal closed contour of $G$-letters contained within $[-n,n]^d$, which contains $[-n(1-\epsilon),n(1-\epsilon)]^d$ with $\Lambda^{\delta}$ probability at least $1 - 2^{-0.5\epsilon n}$ by definition of $\Delta$. Also define $B(u)$ to be the set of sites of $[-n,n]^d$ outside $\gamma(u)$, and $D(u) = u|_{B(u)}$. Then $[\{u \in L_{[-n,n]^d}(X) \ : \ u\delta \in L(X)\}]$ can be written as a disjoint union of sets $[L_{[-n,n]^d}(X) \cap \langle D(u) \rangle]$, meaning that except for a set of $\Lambda^{\delta}$-measure at most $2^{-0.5 \epsilon n}$, $\Lambda^{\delta}|_{[-n(1-\epsilon),n(1-\epsilon)]^d}$ can be written as a weighted average of $\mu\left([x|_{[-n(1-\epsilon),n(1-\epsilon)]^d}] \ | \ [D(u)]\right)$ over possible values of $D(u)$. Since $\mu$ is an MRF, this is in fact a weighted average of $\Lambda^{\gamma(u)}|_{[-n(1-\epsilon),n(1-\epsilon)]^d}$. Finally, by Fact~\ref{unlikelydisfact}, for any $\gamma(u) \neq \gamma(u')$ containing $[-n(1-\epsilon),n(1-\epsilon)]^d$,
\[
d\left(\Lambda^{\gamma(u)}|_{[-n(1-2\epsilon),n(1-2\epsilon)]^d}, \Lambda^{\gamma(u')}|_{[-n(1-2\epsilon),n(1-2\epsilon)]^d}\right) < Z (2n)^d 2^{-\frac{n\epsilon}{7}}.
\]

So, for each $\delta \in \Delta$, except for a set of measure at most $2^{-0.5 \epsilon n}$, $\Lambda^{\delta}|_{[-n(1-2\epsilon),n(1-2\epsilon)]^d}$ can be decomposed as a weighted average of $\Lambda^{\gamma(u)}|_{[-n(1-2\epsilon),n(1-2\epsilon)]^d}$, and each pair $\Lambda^{\gamma(u)}|_{[-n(1-2\epsilon),n(1-2\epsilon)]^d}$, $\Lambda^{\gamma(u')}|_{[-n(1-2\epsilon),n(1-2\epsilon)]^d}$ has total variational distance less than $Z(2n)^d 2^{-\frac{n\epsilon}{7}}$ for some universal constant $Z$. Therefore, for any $\delta, \delta' \in \Delta$, the pair $\Lambda^{\delta}|_{[-n(1-2\epsilon),n(1-2\epsilon)]^d}$, $\Lambda^{\delta'}|_{[-n(1-2\epsilon),n(1-2\epsilon)]^d}$ has total variational distance less than $2^{-0.5\epsilon n} + Z(2n)^d 2^{-\frac{n\epsilon}{7}}$. 

Recall that except for a set of $\mu$-measure at most $2^{-0.5\epsilon n}$, $\mu|_{[-n(1-2\epsilon),n(1-2\epsilon)]^d}$ can be decomposed as a weighted average over $\Lambda^{\delta}|_{[-n(1-2\epsilon),n(1-2\epsilon)]^d}$ for $\delta \in \Delta$. This means that for any $\delta \in \Delta$, 
\[
d\left(\mu|_{[-n(1-2\epsilon),n(1-2\epsilon)]^d}, \Lambda^{\delta}|_{[-n(1-2\epsilon),n(1-2\epsilon)]^d}\right) < 2\cdot 2^{-0.5\epsilon n} + Z(2n)^d 2^{-\frac{n\epsilon}{7}}.
\]
As $n \rightarrow \infty$, the right-hand side of this inequality approaches $0$, and $\mu(\Delta)$ approaches $1$. We have then verified (\ref{QWB}), so $\mu$ is quite weak Bernoulli, and therefore isomorphic to a Bernoulli measure.

\end{proof}

\section{Unrelatedness of the $\epsilon$-fullness condition to mixing conditions}
\label{unrelated}

Often in $\mathbb{Z}^d$ symbolic dynamics, useful properties of an SFT follow from some sort of uniform topological mixing condition, such as block gluing or uniform filling. Examples of such properties are being a measure-theoretically universal model (follows from UFP by \cite{RS}), the existence of dense periodic points (follows from block gluing by the argument in \cite{ward}), and entropy minimality (nonexistence of proper subshift of full entropy; follows from UFP by \cite{QS}). In this section, we will give some examples to show that $\epsilon$-fullness and uniform topological mixing properties are quite different notions for nearest neighbor $\mathbb{Z}^d$ SFTs.

\begin{example}[Non-topologically mixing]\label{notmixing}
Clearly $\epsilon$-fullness implies no mixing conditions at all; the SFT from Example~\ref{onepoint} can be made $\epsilon$-full for arbitrarily small $\epsilon$ by increasing the parameter $n$, but is never topologically mixing since there are no points which contain both a $0$ and a $1$. 
\end{example}

\begin{example}[Topologically mixing but not block gluing]\label{mixingnotgluing}

In \cite{QS}, a $\mathbb{Z}^2$ SFT called the \textbf{checkerboard island shift} is defined. We briefly describe its properties here, but refer the reader to \cite{S} for more details. The checkerboard island shift $C$ is defined by a set of legal $2 \times 2$ configurations, namely those appearing in Figure~\ref{checkerboard}, plus the $2 \times 2$ configuration of all blank symbols. Note that $C$ is not a nearest neighbor SFT; we will deal with this momentarily. It is shown in \cite{QS} that $C$ is topologically mixing, and in fact more is observed; any finite configuration $w \in L(C)$ is a subconfiguration of a configuration $w' \in L(C)$ with only blank symbols on the inner boundary. It is also shown in \cite{QS} that $C$ does not have the uniform filling property. In fact, their proof also shows that $C$ is not block gluing; they observe that any square checked configuration surrounded by arrows (e.g. the central $8 \times 8$ block of Figure~\ref{checkerboard}) forces a square configuration containing it of almost twice the size (e.g. the $14 \times 14$ configuration in Figure~\ref{checkerboard}), which clearly precludes block gluing. 

\begin{figure}[h]
\centering
\includegraphics[scale=0.4]{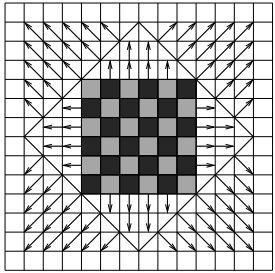}
\caption{A sample configuration from $C$}
\label{checkerboard}
\end{figure}


First, we define a version of $C$ which is nearest neighbor by passing to the second higher block presentation: define $C'$ to have alphabet $A'$ consisting of the $2 \times 2$ configurations from $L(C)$, and the only adjacency rule is that adjacent $2 \times 2$ blocks must agree on the pair of letters along their common edge. (For instance, for letters $a,b,c,d$ in the alphabet of $C$, $\begin{smallmatrix} a & b\\ c & d \end{smallmatrix} \ \begin{smallmatrix} b & c\\ d & a \end{smallmatrix}$ would be a legal adjacent pair of letters in $C'$.) $C'$ is topologically conjugate to $C$, and so shares all properties of $C$ described above. The reader can check that the alphabet $A'$ of $C'$ has size $79$.

We can now make versions of $C'$ which are $\epsilon$-full for small $\epsilon$; for any $N$, define $C'_N$ to have alphabet $A'_N = G \cup A'$, where $G$ is a set of $N$ ``free'' symbols with the following adjacency rules: each $G$-letter can appear next to any other $G$-letter, and each $G$-letter can legally appear next to any letter from $A'$ consisting of four blank symbols from the original alphabet of $C$. The reader may check that the addition of these new symbols does not affect the above arguments proving topological mixing and absence of block gluing, and so each $C'_N$ is topologically mixing, but not block gluing. Also, since any $G$-letter can be legally followed in any direction by any other $G$-letter, $C'_N$ is $\frac{79}{N+79}$-full. Clearly $C'_N$ can then be made $\epsilon$-full for arbitrarily small $\epsilon$ by increasing the parameter $N$.

This example can be trivially extended to a nearest neighbor $\mathbb{Z}^d$ SFT $C'^{(d)}_N$ by keeping the same alphabet and adjacency rules and adding no transition rules along the extra $d-2$ dimensions. Clearly $C'^{(d)}_N$ is still topologically mixing and not block gluing, and can be made $\epsilon$-full for arbitrarily small $\epsilon$ by taking $N$ large.

\end{example}

\begin{example}[Block gluing but not uniform filling]\label{gluingnotUFP}

In \cite{S}, a nearest neighbor $\mathbb{Z}^2$ SFT called the \textbf{wire shift} is defined. We briefly describe its properties here, but refer the reader to \cite{S} for more details. For any integer $N$, the wire shift $W_N$ has alphabet $A_N = G \cup B$, where $B$ consists of six ``grid symbols'' illustrated in Figure~\ref{wiresymb} and $G$ is a set of ``blank tiles'' labeled with integers from $[1,N]$. The adjacency rules are that neighboring letters must have edges which ``match up'' in the sense of Wang tiles; for instance, the leftmost symbol from Figure~\ref{wiresymb} could not appear immediately above the second symbol from the left. 

\begin{figure}[h]
\centering
\includegraphics[scale=0.4]{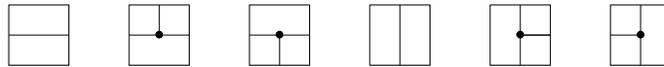}
\caption{The symbols $B$ from the alphabet of $W_N$}
\label{wiresymb}
\end{figure}

It is shown in Corollary 3.3 from \cite{S} that $W_N$ is block gluing for any $N$. In Lemma 3.4 from \cite{S}, it is shown that for any $N \geq 2$, $W_N$ is not entropy minimal, in other words $W_N$ contains a proper subshift with topological entropy $h(W_N)$. However, Lemma 2.7 from \cite{S} shows that any $\mathbb{Z}^d$ SFT with the UFP is entropy minimal, and so $W_N$ does not have the UFP for any $N \geq 2$. 

Again any letter of $G$ can be legally followed in any direction by any other letter of $G$, implying that $W_N$ is $\frac{6}{N+6}$-full. Clearly then $W_N$ can be made $\epsilon$-full for arbitrarily small $\epsilon$ by increasing the parameter $N$.

This example can also be trivially extended to a nearest neighbor $\mathbb{Z}^d$ SFT $W^{(d)}_N$ by keeping the same alphabet and adjacency rules and adding no transition rules along the extra $d-2$ dimensions. Clearly $W^{(d)}_N$ is still block gluing and does not have the UFP, and can be made $\epsilon$-full for arbitrarily small $\epsilon$ by taking $N$ large.

\end{example}

\begin{example}[Uniform filling]\label{UFP}
Any full shift is $\epsilon$-full for any $\epsilon > 0$, and obviously has the UFP.
\end{example}

\begin{remark}\label{supportnotgluing}
We note that we can say a bit more about the extended checkerboard island shift $C'_N$ of Example~\ref{mixingnotgluing} for large $N$. As we already noted, it was shown in \cite{QS} that every $w \in L(C)$ can be extended to a configuration $w' \in L(C)$ with only blank symbols from the alphabet of $C$ on the inner boundary. Then, for any $N$, we claim that any configuration $w \in L(C'_N)$ can be extended to a configuration $w' \in L(C'_N)$ with only $G$-letters on the inner boundary: any configuration $w \in L(C'_N)$ looks like a recoded version of a configuration from $C$, possibly with some $G$-letters replacing letters of $A'$ consisting of four blanks from the original alphabet of $A$, and so the same extension proved in \cite{QS} guarantees that $w$ can be extended to $w'$ with only $G$-letters and four-blank letters from $A'$, which can itself be surrounded by a boundary of $G$-letters. If we take any $N$ for which $\frac{79}{N+79} < \epsilon_d$ and denote the unique measure of maximal entropy on $C'_N$ by $\mu$, then by Corollary~\ref{gwords}, $\mu([w']) > 0$, clearly implying that $\mu([w]) > 0$. We have then shown that $\mu$ has full support, and in particular have shown the following:

\begin{theorem}
For any $\epsilon > 0$, there exists an $\epsilon$-full nearest neighbor $\mathbb{Z}^2$ SFT with unique measure of maximal entropy $\mu$ whose support is not block gluing.
\end{theorem}

In other words, there really is no uniform mixing condition implied by the $\epsilon$-fullness property, even hidden within the support of the unique measure of maximal entropy $\mu$.

\end{remark}

\section{Comparison with existing sufficient conditions for uniqueness of measure of maximal entropy}
\label{comparison}

We for now focus on property (A) from Theorem~\ref{mainresult}, i.e. the fact that for any $d$ and small enough $\epsilon$, any $\epsilon$-full nearest neighbor $\mathbb{Z}^d$ SFT has a unique measure of maximal entropy. In this section, we will attempt to give proper context by giving some examples of conditions in the literature related to our condition.
We first need a definition.

\begin{definition}\label{safe}
For a nearest neighbor $\mathbb{Z}^d$ SFT $X$, a letter $a$ of the alphabet $A$ is a {\bf safe symbol} if $a$ is a legal neighbor of every letter of $A$ in every cardinal direction $\pm \vec{e_i}$.
\end{definition}

\begin{example}\label{markleypaul}
In \cite{MP}, the classical Dobrushin uniqueness criterion for Markov random fields was used to prove the following result:

\begin{theorem}[\cite{MP}, Proposition 5.1]\label{MPunique}
For any nearest neighbor $\mathbb{Z}^d$ SFT $X$ with alphabet $A$ such that at least $|A|\left(\frac{2d}{\sqrt{4d^2 + 1} + 1}\right)$ of the letters of $A$ are safe symbols, $X$ has a unique measure of maximal entropy.
\end{theorem}

This seems to be the first result to show that a large proportion of safe symbols is enough to guarantee uniqueness.

\end{example}

\begin{example}\label{burtonsteif}
In \cite{BS1}, methods from percolation theory, following techniques from \cite{vdB}, were used to prove the following result.

\begin{theorem}[\cite{BS1}, Theorem 1.17]\label{BSunique}
For any nearest neighbor $\mathbb{Z}^d$ SFT $X$ with alphabet $A$ such that at least $|A|(\sqrt{1 - p_c(\mathbb{Z}^d)})$ of the letters of $A$ are safe symbols, $X$ has a unique measure of maximal entropy, where $p_c(\mathbb{Z}^d)$ is the critical probability for site percolation on the $d$-dimensional hypercubic lattice.
\end{theorem}

(This is not the precise statement of their theorem, but is an equivalent reformulation which better contrasts with Theorem~\ref{MPunique}.) We will not define $p_c(\mathbb{Z}^d)$ or discuss percolation theory here; for a good introduction to the subject, see \cite{grimmett}.

\end{example}

Theorem~\ref{BSunique} is stronger than Theorem~\ref{MPunique} for $d = 2$: it is known that $p_c(\mathbb{Z}^2) > 0.5$, so $\sqrt{1 - p_c(\mathbb{Z}^2)} < \sqrt{0.5} < \frac{4}{\sqrt{17} + 1}$. However, for large $d$, $p_c(\mathbb{Z}^d) = \frac{1 + o(1)}{2d}$ (\cite{K}), therefore $\frac{2d}{\sqrt{4d^2 + 1} + 1} = 1 - \frac{1 + o(1)}{2d} < 1 - \frac{1 + o(1)}{4d} = \sqrt{1 - p_c(\mathbb{Z}^d)}$, implying that Theorem~\ref{MPunique} is stronger for large $d$.

\begin{example}\label{haggstrom}
A slightly more general result comes from \cite{H}, which requires another definition.

\begin{definition}
For any $\mathbb{Z}^d$ subshift $X$ with alphabet $A$, the \textbf{generosity of $X$} is 
\[
G(X) = \frac{1}{|A|} \min_{\delta} |\{a \in A \ : \ a\delta \in L(X)\}|,
\]
where the minimum ranges over $\delta \in A^{\mathbb{Z}^d \setminus \{0\}}$ s.t. $\delta a \in X$ for at least one $a \in A$.
\end{definition}

\begin{theorem}[\cite{H}, Theorem 1.12]\label{Hunique}
Any nearest neighbor $\mathbb{Z}^d$ SFT with generosity at least $\frac{1}{1+p_c(\mathbb{Z}^d)}$ has a unique measure of maximal entropy.
\end{theorem}


The strength of Theorem~\ref{Hunique} is that it allows one to consider SFTs without safe symbols. For instance, the {\bf $n$-checkerboard} $\mathbb{Z}^d$ SFT, defined by alphabet $A = \{1,\ldots,n\}$ and the adjacency rule that no letter may be adjacent to itself in any cardinal direction, has generosity $1 - \frac{2d}{n}$, which satisfies the hypotheses of Theorem~\ref{Hunique} for large $n$. However, it has no safe symbols, and so cannot satisfy the hypotheses of Theorems~\ref{MPunique} and ~\ref{BSunique}.

\end{example}

\begin{remark}
We note that Haggstrom also showed in \cite{H} that it is not possible for any $d \geq 2$ to give a lower bound on $G(X)$ in Theorem~\ref{Hunique} which would imply uniqueness of measure of maximal entropy for all $\mathbb{Z}^d$ SFTs (not just nearest neighbor): Theorem 1.13 from \cite{H} states that for any $d \geq 2$ and any $\epsilon > 0$, there exists a $\mathbb{Z}^d$ SFT $X$ with more than one measure of maximal entropy such that $G(X) \geq 1-\epsilon$. This implies that such a uniform lower bound would be impossible for part (A) of our Theorem~\ref{mainresult} as well, since any nearest neighbor $\mathbb{Z}^d$ SFT with generosity more than $1 - \epsilon$ is clearly $\epsilon$-full. 

\end{remark}

One property shared by all of each of these conditions is that they require all letters of $A$ to satisfy some fairly stringent adjacency properties. For instance, if $A$ contains even one letter which has only a single allowed neighbor in some direction, then it has at most one safe symbol and its generosity equals $\frac{1}{|A|}$ (the minimum possible amount). The strength of the $\epsilon$-fullness condition is that it allows the existence of a small set of letters with bad adjacency properties, as long as the rest of the symbols are ``close enough'' to being safe symbols. 

\section{Optimal value of $\beta_d$}
\label{critical}



It is natural to define the optimal value of $\beta_d$ in Theorem~\ref{mainresult2} which guarantees uniqueness of the measure of maximal entropy.

\begin{definition}
For every $d$, we define
\begin{multline*}
\alpha_d := \inf \{\alpha \ : \ \exists \textrm{a nearest neighbor $\mathbb{Z}^d$ SFT $X$}\\ \textrm{with more than one measure of maximal entropy for which $h(X) \geq (\log |A|) - \alpha$.}\}
\end{multline*}
\end{definition}

We can determine $\alpha_1$ exactly.

\begin{theorem}\label{1Dalpha}
$\alpha_1 = \log 2$.
\end{theorem}

\begin{proof}
Suppose that $X$ is a nearest neighbor $\mathbb{Z}$ SFT with alphabet $A$ and more than one measure of maximal entropy. Since any measure of maximal entropy of $X$ is supported in an irreducible component of $X$, clearly $X$ has at least two irreducible components of entropy $h(X)$. (See \cite{LM} for the definition of irreducible components and a simple proof of this fact.) Each of these components must then have alphabet size at least $e^{h(X)}$, and so $|A| \geq 2e^{h(X)}$. This implies that $(\log |A|) - h(X) \geq \log 2$, and so $\alpha_1 \geq \log 2$.

However, the same idea shows that there exists a nearest neighbor $\mathbb{Z}$ SFT $X$ with multiple measures of maximal entropy and for which $(\log |A|) - h(X) = \log 2$; if $X$ is the union of two disjoint full shifts on $n$ symbols, then each of the full shifts supports a measure of maximal entropy for $X$. Therefore, $\alpha_1 \leq \log 2$, completing the proof. 

\end{proof}

We will now show that $\alpha_2 < \log 2$, meaning that in two dimensions, it is possible to have an SFT $X$ in which two disjoint portions of the alphabet induce different measures of maximal entropy and can still coexist within the same point of $X$ (unlike the one-dimensional case).

\begin{theorem}\label{iceberg}
$\alpha_2 < \log 2$.
\end{theorem}

\begin{proof}

In \cite{BS0}, an example is given of a strongly irreducible nearest neighbor $\mathbb{Z}^2$ SFT with two ergodic measures of maximal entropy, later called the \textbf{iceberg model} in the literature. We quickly recall the definition of this SFT. 

The iceberg model, which we denote by $\mathcal{I}_M$, is defined by a positive integer parameter $M$. The alphabet is $A_M = \{-M, \ldots, -2, -1, 1, 2, \ldots, M\}$. The adjacency rules of $\mathcal{I}_M$ are that any letters with the same sign may neighbor each other, but a positive may only sit next to a negative if they are $1$ and $-1$. It is shown in \cite{BS0} that for any $M > 4e 28^2$, $\mathcal{I}_M$ has exactly ergodic two measures of maximal entropy. 

We will now show that $h(\mathcal{I}_M)$ is strictly greater than $\log M$, which will imply that $(\log |A_M|) - h(\mathcal{I}_M) < \log 2$. For any $n$, define the set $P_n$ of all configurations with shape $[1,n]^2$ consisting of letters from $\{1,\ldots,M\}$. Clearly $P_n \subset L_{[1,n]^d}(\mathcal{I}_M)$. By ergodicity of the Bernoulli measure giving each positive letter equal probability (or simply the Strong Law of Large Numbers), if we define $G_n$ to be the set of configurations in $P_n$ with at least $\frac{n^2}{M^5}$ occurrences of the configuration $\begin{smallmatrix} & 1 & & \\ 1 & 1 & 1\\ & 1 & & \end{smallmatrix}$, then $\lim_{n \rightarrow \infty} \frac{\log |G_n|}{n^2} = \log M$. In any $u \in G_n$, it is simple to choose at least $\frac{n^2}{5M^5}$ occurrences of $\begin{smallmatrix} & 1 & & \\ 1 & 1 & 1\\ & 1 & & \end{smallmatrix}$ with disjoint centers. Then, one can construct a set $f(u)$ of at least $2^{\frac{n^2}{5M^5}}$ configurations in $L_{[1,n]^2}(\mathcal{I}_M)$ by independently either replacing the center of each of these $1$-crosses by $-1$ or leaving it as a $1$. It is easy to see that $f(u) \cap f(u') = \varnothing$ for any $u \neq u' \in G_n$. Therefore, $|L_{[1,n]^2}(\mathcal{I}_M)| \geq |\bigcup_{u \in G_n} f(u)| \geq 2^{\frac{n^2}{5M^5}} |G_n|$, implying that
\[
h(\mathcal{I}_M) = \lim_{n \rightarrow \infty} \frac{\log |L_{[1,n]^2}(\mathcal{I}_M)}{n^2} \geq \lim_{n \rightarrow \infty} \frac{\log |G_n|}{n^2} + \frac{\log 2}{5M^5} = \log M + \frac{\log 2}{5M^5}.
\]
Therefore, $(\log|A_M|) - h(\mathcal{I}_M) \leq \log 2 - \frac{\log 2}{5M^5}$, and so since $\mathcal{I}_M$ has two measures of maximal entropy for $M = 10000 > 4e 28^2$, $\alpha_2 \leq \log 2 - \frac{\log 2}{5 \cdot 10^{20}}$.

\end{proof}

Unsurprisingly, the sequence $\alpha_d$ is nonincreasing in $d$.

\begin{theorem}\label{noninc}
$\alpha_{d+1} \leq \alpha_d$ for all $d$.
\end{theorem}

\begin{proof}

Fix any $d$ and any $\epsilon > 0$. By definition, there exists a nearest neighbor $\mathbb{Z}^d$ SFT $X$ (with alphabet $A$) with $\mu_1 \neq \mu_2$ measures of maximal entropy for which $h(X) > (\log |A|) - \alpha_d - \epsilon$. We may then define $X^{\mathbb{Z}}$ to be the nearest neighbor $\mathbb{Z}^{d+1}$ SFT containing all $x \in A^{\mathbb{Z}^{d+1}}$ for which each $x|_{\mathbb{Z}^d \times \{j\}} \in X$. In other words, $X^{\mathbb{Z}}$ has the adjacency rules for $X$ in the $\vec{e_1}, \ldots, \vec{e_d}$-directions, and no restrictions at all in the $\vec{e_{d+1}}$-direction.

Then clearly $h(X^{\mathbb{Z}}) = h(X) > (\log |A|) - \alpha_d - \epsilon$. Also, it is not hard to check that $\mu_1^{\mathbb{Z}}$ and $\mu_2^{\mathbb{Z}}$ are measures of maximal entropy for $X^{\mathbb{Z}}$, where $\mu_i^{\mathbb{Z}}$ is the independent product of countably many copies of $\mu_i$ in the $\vec{e_{d+1}}$-direction. Therefore, $\alpha_{d+1} \leq \alpha_d + \epsilon$. Since $\epsilon$ was arbitrary, we are done.

\end{proof}

Theorem~\ref{mainresult2} implies that $\alpha_d$ decays at most polynomially; $\alpha_d \geq \beta_d = \frac{1}{3^6  2^{24} d^{17}}$. Our final result gives an upper bound on $\alpha_d$ which also decays polynomially.

\begin{theorem}\label{alpharate}
There exists a constant $B$ so that $\alpha_d \leq \frac{1}{\lfloor B d^{0.25} (\log d)^{-0.75} \rfloor}$ for all $d$.
\end{theorem}

\begin{proof}
The main tool in our construction is a theorem of Galvin and Kahn (\cite{GK}) about phase transitions for the hard-core shift with activities. 
Specifically, define a Gibbs measure with activity $\lambda \in \mathbb{R}^+$ on the $\mathbb{Z}^d$ hard-core shift $\mathcal{H}_d$ to be a measure $\mu$ with the property that for any $n$ and any configurations $w$ with shape $S$ and $\delta$ with shape $\partial S$ such that $w\delta \in L(\mathcal{H}_d)$,
\[
\mu([w] \ | \ [\delta]) = \frac{\lambda^{\# \textrm{ ones in } w}}{\sum_{v \in \{0,1\}^{S} \textrm{ s.t. } v\delta \in L(\mathcal{H}_d)} \lambda^{\# \textrm{ ones in } v}}.
\]
In other words, the conditional probability of $w$ given a fixed $\delta$ is proportional to $\lambda^{\# \textrm{ ones in } w}$. One way to create Gibbs measures is to fix any boundary conditions $\delta_n \in L_{\partial [-n,n]^d}(\mathcal{H}_d)$ on larger and larger cubes $[-n,n]^d$, and take a weak limit point of the sequence of conditional measures $\mu([x|_{[-n,n]^d}] \ | \ [\delta])$. For the hard-core model, two boundary conditions of interest are $\delta_{e,n}$, which contains $1$ on all odd $t \in \partial [-n,n]^d$ ($\sum t_i$ odd) and $0$ on all even $t \in \partial [-n,n]^d$  ($\sum t_i$ even), and $\delta_{o,n}$, which does the reverse.

The main result of \cite{GK} states that there exists a universal constant $C$ with the following property: for any dimension $d$ and activity level $\lambda$ which is greater than $C d^{-0.25} (\log d)^{0.75}$, the sequences of conditional measures $\mu([x_{[-n,n]^d}] \ | \ [\delta_{e,n}])$ and $\mu([x_{[-n,n]^d}] \ | \ [\delta_{o,n}])$ approach respective weak limits $\mu_e$ and $\mu_o$, which are distinct Gibbs measures with activity $\lambda$ on $\mathcal{H}_d$. Though it is not explicitly stated in \cite{GK}, it is well-known that each of $\mu_o$ and $\mu_e$ is a shift by one unit in any cardinal direction of the other, and in particular that both $\mu_o$ and $\mu_e$ are invariant under any shift in $(2\mathbb{Z})^d$. (This is mentioned in, among other places, \cite{vdBS}.) The strategy used in \cite{GK} to show that $\mu_o \neq \mu_e$ is quite explicit: it is shown that for $\lambda$ satisfying the hypothesis of the theorem, $\mu_e(x(0) = 1) < \mu_o(x(0) = 1)$. 

We now define $\mathcal{H}_{N,d}$ to be the nearest-neighbor $\mathbb{Z}^d$ SFT with $N$ safe symbols $\{0_1, \ldots, 0_N\}$ and a symbol $1$ which cannot appear next to itself in any cardinal direction; $\mathcal{H}_{N,d}$ is a version of the usual hard-core shift where the symbol $0$ has been ``split'' into $N$ copies. For any Gibbs measure $\mu$ on the hard-core model $\mathcal{H}_d$ with activity $\lambda = \frac{1}{N}$, define a measure $\hat{\mu}$ on $\mathcal{H}_{N,d}$ by ``splitting'' the measure of any cylinder set $[w]$ uniformly over all $\displaystyle N^{\# \ \textrm{zeroes in } w}$ ways of assigning subscripts to $0$ symbols in $w$. It is easily checked that any such measure $\hat{\mu}$ has the uniform conditional probabilities property from the conclusion of Proposition~\ref{P1} (Proposition 1.20 from \cite{BS1}), which stated that all measures of maximal entropy have this property. In fact, Proposition 1.21 from \cite{BS1} gives a partial converse: for strongly irreducible SFTs, any shift-invariant measure with uniform conditional probabilities must be a measure of maximal entropy.

Since $\mathcal{H}_{N,d}$ is clearly strongly irreducible, this would show that $\widehat{\mu_e}$ and $\widehat{\mu_o}$ are measures of maximal entropy on $\mathcal{H}_{N,d}$, were it not for the fact that these measures are not shift-invariant. However, their average $\frac{1}{2}(\widehat{\mu_e} + \widehat{\mu_o})$ clearly shares the uniform conditional probability property, and is shift-invariant, and so is a measure of maximal entropy on $\mathcal{H}_{N,d}$. In fact, it is the unique measure of maximal entropy on $\mathcal{H}_{N,d}$. We will show, however, that the direct product of $\mathcal{H}_{N,d}$ with itself can have multiple measures of maximal entropy.

Define the nearest neighbor $\mathbb{Z}^d$ SFT $\mathcal{H}_{N,d}^2$ with alphabet $\{(a,b) \ : \ a,b \in \{0_1,\ldots,0_N,1\}\}$, where the adjacency rules from $\mathcal{H}_{N,d}$ are separately enforced in each coordinate. In other words, $(0_1, 1)$ may appear next to $(1, 0_2)$ since $0_1 1$ and $1 0_2$ are each legal in $\mathcal{H}_{N,d}$, but $(0_1, 1)$ cannot appear next to $(1,1)$: though $0_1 1$ is legal in $\mathcal{H}_{N,d}$, $1 1$ is not. Define the measures $\nu_1 := \frac{1}{2}(\widehat{\mu_o} \times \widehat{\mu_o} + \widehat{\mu_e} \times \widehat{\mu_e})$ and $\nu_2 := \frac{1}{2}(\widehat{\mu_o} \times \widehat{\mu_e} + \widehat{\mu_e} \times \widehat{\mu_o})$ on $\mathcal{H}_{N,d}^2$. It should be obvious that both $\nu_1$ and $\nu_2$ have the uniform conditional probability property mentioned above and that both are shift-invariant. Since $\mathcal{H}_{N,d}^2$ is strongly irreducible (it still has safe symbols, for instance $(0_1, 0_1)$), both $\nu_1$ and $\nu_2$ are therefore measures of maximal entropy on $\mathcal{H}_{N,d}^2$.

We claim that $\nu_1 \neq \nu_2$ as long as $N < C^{-1} d^{0.25} (\log d)^{-0.75}$. If this condition holds, then $\frac{1}{N} > C d^{-0.25} (\log d)^{0.75}$, and so by \cite{GK}, $\mu_e(x(0) = 1) < \mu_o(x(0) = 1)$ for the hard-core shift with activity $\lambda = \frac{1}{N}$. Clearly, $\widehat{\mu_e}(x(0) = 1) = \mu_e(x(0) = 1)$ and $\mu_o(x(0) = 1) = \widehat{\mu_o}(x(0) = 1)$, so $\widehat{\mu_e}(x(0) = 1) < \widehat{\mu_o}(x(0) = 1)$. For brevity, denote these probabilities by $\alpha$ and $\beta$ respectively. By definition, $\nu_1(x(0) = (1,1)) = \frac{1}{2}(\alpha^2 + \beta^2)$ and $\nu_2(x(0) = (1,1)) = \frac{1}{2}(\alpha \beta + \alpha \beta)$. These are equal if and only if $\alpha = \beta$, which is not the case. Therefore, $\nu_1 \neq \nu_2$ and $\mathcal{H}_{N,d}^2$ has multiple measures of maximal entropy.

Our final observation is that any way of placing arbitrary letters on even sites and only pairs $(0_i, 0_j)$ on odd sites yields configurations in $L(\mathcal{H}_{N,d}^2)$, and so $h(\mathcal{H}_{N,d}^2) \geq \frac{1}{2}(\log (N+1)^2 + \log N^2) = \log N(N+1)$. Therefore, for this SFT, $\log |A| - h(\mathcal{H}_{N,d}^2) \leq \log (N+1)^2 - \log N(N+1) = \log \left( 1 + \frac{1}{N} \right) < \frac{1}{N}$. Choosing $B = C^{-1}$ and $N = \lfloor C^{-1} d^{0.25} (\log d)^{-0.75} \rfloor$ now completes our proof.

\end{proof}

Just the fact that $\alpha_d \rightarrow 0$ provides some quantification of the commonly believed heuristic that it should get easier, not harder, to have multiple measures of maximal entropy as $d \rightarrow \infty$, as there are more paths in which information can communicate. 

Our bounds on $\alpha_d$ are then $O(d^{-17}) < \alpha_d < d^{-0.25 + o(1)}$. We imagine that the true values of $\alpha_d$ are closer to the upper bounds than the lower, but have no conjectures as to the exact rate.

\section*{Acknowledgments}\label{acks}
The author would like to acknowledge the hospitality of the University of British Columbia, where much of this work was completed.

\end{document}